\documentclass[a4paper,reqno, doublespace]{amsart}
\usepackage{amsmath, amssymb}
\usepackage{amsfonts}
\usepackage{graphicx}
\usepackage{color}
\usepackage[usenames,dvipsnames]{xcolor}
\usepackage[english]{babel}
\usepackage{bm}

\newtheorem{teo}{Theorem}[section]

\newtheorem{lemma}{Lemma}[section]
\newtheorem{proposition}{Proposition}[section]

\newtheorem{definition}{Definition}[section]

\numberwithin{equation}{section}

\begin{document}
\title[Injection-suction control]{Injection-suction control for
Navier-Stokes equations with slippage}
\author{N.V. Chemetov}
\address{Universidade de Lisboa, \\
Edificio C6, 1 Piso, Campo Grande, \\
1749-016 Lisboa \\
Portugal }
\email{nvchemetov@fc.ul.pt }
\author{F. Cipriano}
\address{CMA / UNL and Dep. de Matem{\'a}tica, FCT-UNL \\
Universidade Nova de Lisboa, Quinta da Torre \\
2829 -516 Caparica, Lisboa \\
Portugal}

PORTUGAL
\email{cipriano@cii.fc.ul.pt}
\date{}

\begin{abstract}
We consider a velocity tracking problem for the Navier-Stokes equations in a
2D-bounded domain. The control acts on the boundary through a
injection-suction device and the flow is allowed to slip against the surface
wall. We study the well-posedness of the state equations, linearized state
equations and adjoint equations. In addition, we show the existence of an
optimal solution and establish the first order optimality condition.
\end{abstract}

\maketitle




\section{Introduction}

\setcounter{equation}{0}

The goal of this article is to study an optimal boundary control problem for
viscous incompressible fluids, filling a bounded domain $\Omega \subset
\mathbb{R}^{2}$, and governed by the Navier-Stokes equations with
non-homogeneous Navier slip boundary conditions
\begin{equation}
\left\{
\begin{array}{ll}
\partial _{t}\mathbf{y}+\mathrm{div}\,(\mathbf{y}\otimes \mathbf{y})-\nabla
p=\Delta \mathbf{y},\mathbf{\qquad }\mbox{div}\,\mathbf{y}=0, & \mbox{in}{\
\Omega _{T}=(0,T)\times \Omega },\vspace{2mm} \\
\mathbf{y}\cdot \mathbf{n}=a,\;\quad \left[ 2D(\mathbf{y})\,\mathbf{n}%
+\alpha \mathbf{y}\right] \cdot {\bm{\tau }}=b\;\quad & \text{on}\ \Gamma
_{T}=(0,T)\times \Gamma ,\vspace{2mm} \\
\mathbf{y}(0,\mathbf{x})=\mathbf{y}_{0}(\mathbf{x}) & \mbox{in}\ \Omega ,%
\end{array}%
\right.  \label{NSy}
\end{equation}%
\vspace{1pt}where $\mathbf{y}=\mathbf{y}(t,\mathbf{x})$ is the velocity, $%
p=p(t,\mathbf{x})$ is the pressure and the condition verifies
\begin{equation}
\mbox{div}\,\mathbf{y}_{0}=0\qquad \mbox{ in   }\ \Omega .  \label{ICNS}
\end{equation}%
Here $\ D(\mathbf{y})=\frac{1}{2}[\nabla \mathbf{y}+(\nabla \mathbf{y})^{T}]$
is the rate-of-strain tensor; $\mathbf{n}$ is the external unit normal to
the boundary $\Gamma \in C^{2}$ of the domain $\Omega $ and $\bm{\tau }$ is
the tangent unit vector to $\Gamma ,$ such that $(\mathbf{n},\bm{\tau })$
forms a standard orientation in $\mathbb{R}^{2}.$\ The function $\alpha
=\alpha (t,\mathbf{x})$ is a so-called friction coefficient. The quantity $a$
corresponds to inflow and outflow fluid through $\Gamma $, satisfying \ the
natural condition%
\begin{equation}
\int_{\Gamma }a(t,\mathbf{x})\,\,d\mathbf{\gamma }=0\quad \quad
\mbox{ for
any  }\;t\in \lbrack 0,T].  \label{eqC2}
\end{equation}

In the literature, the Navier-Stokes equations are usually studied with the
Dirichlet boundary condition $\mathbf{y}=g$ on $\Gamma _{T}$, however it is
well known that for small values of the viscosity, the Dirichlet boundary
conditions is a source of problems due to the adherence of fluid particles
to the boundary and the creation of a strong boundary layer. The laminar
flow is often disturbed by the boundary layer breaking away from the
surface. This flow separation region results in increased overall drag. On
the other hand, theoretical studies and practical experimental (see \cite%
{buc}, \cite{clop}-\cite{CC6}, \cite{ja}, \cite{prie}, \cite{pri}) emphasize
the importance of the surface roughness on the slip behavior of the fluid
particles on the surface wall. Accordingly, slip type boundary conditions,
which were firstly introduces by Navier in 1823, have renewed interest in
order to describe the physical phenomena is appropriate way.

In this work, we consider a tracking problem with a injection-suction
control through the boundary, by allowing simultaneously the fluid to slip
in a natural way along the boundary, and aim to solve the control problem
and state the first order optimality condition.

Let us mention that boundary control is of main importance in several
branches of the industry, for instance in the aviation industry extensive
research has been carried out concerning the implementation of
injection-suction devices to control the motion of the fluid (see \cite{arn}%
, \cite{bla}, \cite{bra}, \cite{mar}, \cite{shu}).

From the mathematical point of view, the boundary control in general is
technically hard to deal with (see \cite{ghs}, \cite{gm}), in the case of
the slip boundary condition, the tangent component of the velocity field
being part of the solution is not given in advance, which requires a very
careful management of the boundary terms, that appear in the state equation,
linearized state equations as well as in the adjoint equations.

In this article we consider a quadratic cost functional, which depends on
the boundary control variables and with a desired target velocity, and prove
the existence of a optimal control, furthermore, we establish the first
order optimality condition. We recall that the optimality condition is a
very difficult issue when dealing with nonlinear systems, since it requires
the well-posedness of the boundary values problems for the state equation
linearized state equation and the adjoint equation. In addition, we should
verify that the linearized state and the adjoint state are related by a
suitable integration by parts formula.

The plan of the present paper is as follows. In Section \ref{sec0}, we
present the general setting, by introducing the appropriate functional
spaces and some necessary classical inequalities. The formulation of the
problem and the main results are stated in Section \ref{sec1}. Section \ref%
{sec2} deals with the well-posedness of the state equations. In Section \ref%
{sec3}, we show that the control-to-state mapping is Lipschitz continuous.
Section \ref{sec4} is devoted to the well-posedness of the linearized state
equations. In Section \ref{sec5}, we verify that the G\^{a}teaux derivative
of the control-to-state mapping corresponds to the solution of the
linearized state equation. Section \ref{sec6} deals with the formulation of
the adjoint equations and to the study of the existence and uniqueness of
the solutions. In Section \ref{sec7} we deduce the duality relation between
the linearized state and the adjoint state. Finally, in Section \ref{sec8}
we prove the main result of the article, Theorems \ref{main_existence} and %
\ref{main_1}.

\section{General setting}

\bigskip \label{sec0}\setcounter{equation}{0}

We define the spaces%
\begin{eqnarray*}
H &=&\{\mathbf{v}\in L_{2}(\Omega ):\,\mbox{div }\mathbf{v}=0\;\ \text{ in}\;%
\mathcal{D}^{\prime }(\Omega ),\;\ \mathbf{v}\cdot \mathbf{n}=0\;\ \text{ in}%
\;H^{-1/2}(\Gamma )\}, \\
V &=&\{\mathbf{v}\in H^{1}(\Omega ):\,\mbox{div }\mathbf{v}=0\;\;\text{a.e.
in}\;\Omega ,\;\ \mathbf{v}\cdot \mathbf{n}=0\;\ \text{ in}\;H^{1/2}(\Gamma
)\}.
\end{eqnarray*}

In what follows we will frequently \ use the standard inequality
\begin{equation}
uv\leqslant \varepsilon u^{2}+\frac{v^{2}}{4\varepsilon },\quad \quad
\forall \varepsilon >0,  \label{ab}
\end{equation}%
Young's inequality
\begin{equation}
uv\leqslant \frac{u^{p}}{p}+\frac{v^{q}}{q},\quad \quad \frac{1}{p}+\frac{1}{%
q}=1,\quad \forall p,q>1,  \label{yi}
\end{equation}
and the equality
\begin{equation}
-\int_{\Omega }\triangle \mathbf{v}\cdot \boldsymbol{\psi }\ d\mathbf{x}%
=-\int_{\Gamma }\left( 2D(\mathbf{v})\mathbf{n}\right) \cdot \boldsymbol{%
\psi }\,\ d\mathbf{\gamma }+\int_{\Omega }2\,D(\mathbf{v}):D(\boldsymbol{%
\psi })\,d\mathbf{x},  \label{integrate}
\end{equation}%
which is valid for any $\mathbf{v}\in H^{2}(\Omega )\cap V$\ and $%
\boldsymbol{\psi }\in H^{1}(\Omega ).$

\bigskip

The following results are well-known, and can be found on the pages 62, 69
of \cite{lad}, p. 125 of \cite{nir}, Lemma 2 of \cite{S73} and \cite{ol}.

\begin{lemma}
\label{gag} Let us denote by $\mathbf{v}_{\Omega }=\int_{\Omega }\mathbf{v\ }%
d\mathbf{x}.$\ For any $\forall \mathbf{v}\in H^{1}(\Omega )$ \ the
Gagliardo--Nirenberg-Sobolev%
\begin{equation}
||\mathbf{v}-\mathbf{v}_{\Omega }||_{L_{q}(\Omega )}\leqslant C||\mathbf{v}%
||_{L_{2}(\Omega )}^{2/q}||\nabla \mathbf{v}||_{L_{2}(\Omega )}^{1-2/q},%
\mathbf{\quad }\forall q\geq 2,  \label{LI}
\end{equation}%
the trace interpolation inequality%
\begin{equation}
||\mathbf{v}-\mathbf{v}_{\Omega }||_{L_{2}(\Gamma )}\leqslant C||\mathbf{v}%
||_{L_{2}(\Omega )}^{1/2}||\nabla \mathbf{v}||_{L_{2}(\Omega )}^{1/2}
\label{TT}
\end{equation}%
are valid.

Moreover if $\mathbf{v\in }V$ satisfies the Navier boundary condition $\left[
2D(\mathbf{v})\,\mathbf{n}+\alpha \mathbf{v}\right] \cdot {\bm{\tau }}=0$ on
the boundary $\Gamma $ with $\alpha \neq 0,$ then Korn's inequality
\begin{equation}
\left\Vert \mathbf{v}\right\Vert _{H^{1}}\leqslant C\left\Vert D(\mathbf{v}%
)\right\Vert _{L_{2}(\Omega )}  \label{Korn}
\end{equation}%
is also valid. Here the constants $C$ depend only on the domain $\Omega .$
\end{lemma}

\bigskip

\bigskip We notice that any vector $\mathbf{v}\in V$\ satisfies the
condition $\mathbf{v}_{\Omega }=0,$\ since
\begin{equation*}
\int_{\Omega }v_{j}\mathbf{\ }d\mathbf{x}=\int_{\Omega }div(\mathbf{v}x_{j})%
\mathbf{\ }d\mathbf{x}=\int_{\Gamma }x_{j}(\mathbf{v}\cdot \mathbf{n)\ }d%
\mathbf{\gamma }=0\qquad \text{for}\mathit{\ \ \ }j=1,2.
\end{equation*}%
We should mention that as in the previous Lemma as well as throughout the
article, we will represent by $C$ a generic constant that can assume
different values from line to line.

\bigskip

Let us define the space $C([0,T];L_{2}(\Omega ))$ of continuous functions on
$[0,T]$ with values in $L_{2}(\Omega ),$ \ endowed by the norm $%
||v||_{C([0,T];L_{2}(\Omega ))}=\max_{t\in \lbrack
0,T]}||v(t)||_{L_{2}(\Omega )}$ and the space
\begin{equation*}
\mathcal{W}(0,T;\Omega )=L_{2}(0,T;H^{1}(\Omega ))\cap
H^{1}(0,T;H^{-1}(\Omega )),
\end{equation*}%
provided with the norm
\begin{equation*}
||v||_{\mathcal{W}(0,T;\Omega )}=||v||_{L_{2}(0,T;H^{1}(\Omega
))}+||v||_{H^{1}(0,T;H^{-1}(\Omega ))}.
\end{equation*}%
We remember the following interpolation result, given in \ \cite{LM} (see
Proposition 3.1, p. 18 and Theorem 3.1, p. 125).

\begin{lemma}
\label{LM} The embedding
\begin{equation*}
\mathcal{W}(0,T;\Omega )\hookrightarrow C([0,T];L_{2}(\Omega ))
\end{equation*}%
is a continuous and linear mapping, that is there exists a constant $C$,
depending only on $\Omega ,$ such that%
\begin{equation*}
||v||_{C([0,T];L_{2}(\Omega ))}\leqslant C||v||_{\mathcal{W}(0,T;\Omega
)}\qquad \text{for any }v\in \mathcal{W}(0,T;\Omega ).
\end{equation*}
\end{lemma}

\bigskip

Finally, for \ $p\in (2,+\infty )$ \ let us set the space
\begin{equation*}
\mathcal{H}_{p}(0,T;\Gamma )=\left( H^{1}(0,T;H^{-\frac{1}{2}}(\Gamma ))\cap
L_{2}(0,T;W_{p}^{1-\frac{1}{p}}(\Gamma ))\right) \times L_{2}(\Gamma _{T}),
\end{equation*}%
endowed with the norm%
\begin{equation*}
||(a,b)||_{\mathcal{H}_{p}(0,T;\Gamma )}=||a||_{L_{2}(0,T;W_{p}^{1-\frac{1}{p%
}}(\Gamma ))}+\,||\partial _{t}a||_{L_{2}(0,T;W_{2}^{-\frac{1}{2}}(\Gamma
))}+\Vert b\Vert _{L_{2}(\Gamma _{T})}.
\end{equation*}

\bigskip

In this work we consider the data $a,b,\alpha $ and $\mathbf{v}_{0}$\ in the
following Banach spaces%
\begin{eqnarray}
(a,b) &\in &\mathcal{H}_{p}(0,T;\Gamma )\quad \quad \text{\textit{for given}}%
\;p\in (2,+\infty ),  \notag \\
\alpha &\in &L_{\infty }(\Gamma _{T})\cap H^{1}(0,T;L_{\infty }(\Gamma
)),\quad \mathbf{v}_{0}\in H.  \label{eq00sec12}
\end{eqnarray}

\bigskip

\section{Formulation of the problem and main results}

\label{sec1}\setcounter{equation}{0}

The main goal of this paper is to control the solution of the system (\ref%
{NSy}) by a boundary control $(a,b)$, which belongs to the space of
admissible controls $\mathcal{A}$ that is defined as a bounded and convex
subset of $\mathcal{H}_{p}(0,T;\Gamma ).$

The cost functional is given by
\begin{equation}
\displaystyle J(a,b,\mathbf{y})=\frac{1}{2}\int_{\Omega _{T}}|\mathbf{y}-%
\mathbf{y}_{d}|^{2}\,d\mathbf{x}dt+\int_{\Gamma _{T}}\left( \frac{\lambda
_{1}}{2}|a|^{2}+\frac{\lambda _{2}}{2}|b|^{2}\right) \,d\mathbf{\gamma }dt
\label{cost}
\end{equation}%
where $\mathbf{y}_{d}\in L_{2}(\Omega _{T})$ is a desired target field and $%
\lambda _{1},\lambda _{2}\geq 0.$ We aim to control the solution $\mathbf{y}$
minimizing the cost functional (\ref{cost}) for an appropriate $(a,b)\in
\mathcal{A}$. More precisely, our goal is to solve the following problem
\vspace{1mm}
\begin{equation*}
(\mathcal{P})\left\{
\begin{array}{l}
\underset{(a,b)}{\mbox{minimize}}\{J(a,b,\mathbf{y}):~(a,b)\in \mathcal{A}%
\quad \text{and}\quad  \\
\qquad \mathbf{y}\mbox{  is  the solution of the  system }\eqref{NSy}%
\mbox{  for
the minimizing   }(a,b)\in \mathcal{A}\}.%
\end{array}%
\vspace{3mm}\quad \right.
\end{equation*}%
\vspace{1mm}

The first main result of this article establishes the existence of solution
for the control problem $(\mathcal{P})$

\begin{teo}
\label{main_existence} Let $\mathcal{A}$ be a bounded convex subset of $%
\mathcal{H}_{p}(0,T;\Gamma ).$ Then there exists at least one solution for
the problem $(\mathcal{P}).$
\end{teo}

\bigskip

Now we give the formulation of the second main result which deals with first
order necessary optimality condition for the problem $(\mathcal{P})$.

\begin{teo}
\label{main_1} Assume that $(a^{\ast },b^{\ast },\mathbf{y}^{\ast })$ is a
solution of the problem $(\mathcal{P})$. In addition assume that $a^{\ast }$
belongs to $H^{1}(0,T;L_{\infty }(\Gamma ))$. Then there exists a unique
solution
\begin{equation*}
\mathbf{p}^{\ast }\in C([0,T];L_{2}(\Omega ))\cap L_2(0,T;H^{2}(\Omega
)),\qquad \pi^*\in L_2(0,T;H^{1}(\Omega ))
\end{equation*}%
of the adjoint system
\begin{equation}
\left\{
\begin{array}{ll}
-\partial _{t}\mathbf{p}^{\ast }-2D(\mathbf{p}^{\ast })\,\mathbf{y}^{\ast
}+\nabla \pi^* =\Delta \mathbf{p}^{\ast }+\left( \mathbf{y}^{\ast }-\mathbf{y%
}_{d}\right) , &  \\
\mathrm{div}\,\mathbf{p}^{\ast }=0 & \quad \mbox{in}\ \Omega _{T},\vspace{2mm%
} \\
\mathbf{p}^{\ast }\cdot \mathbf{n}=0,\qquad \left[ 2D(\mathbf{p}^{\mathbf{%
\ast }})\mathbf{n}+(a+\alpha )\mathbf{p}^{\ast }\right] \cdot \mathbf{%
\bm{\tau }}=0 & \quad \mbox{on}\ \Gamma _{T},\vspace{2mm} \\
\mathbf{p}^{\ast }(T)=0 & \quad \mbox{in}\ \Omega ,%
\end{array}%
\right.  \label{MR2}
\end{equation}%
verifying the optimality condition
\begin{eqnarray}
&&\int_{\Gamma _{T}}\{(f-a^{\ast })[\pi^* +(\mathbf{p}^{\ast }\cdot \mathbf{y%
}^{\ast })+\left( 2D(\mathbf{p}^{\mathbf{\ast }})\mathbf{n}\right) \cdot
\mathbf{n}]  \notag \\
&&+(b^{\ast }-g)(\mathbf{p}^{\ast }\cdot \mathbf{\bm{\tau })}+\lambda
_{1}a^{\ast }\left( a^{\ast }-f\right) \,+\lambda _{2}b^{\ast }\left(
b^{\ast }-g\right) \,\}~d\mathbf{\gamma }dt\geq 0  \label{zvMR2}
\end{eqnarray}%
for all $(f,g)\in \mathcal{H}_{p}(0,T;\Gamma )$.
\end{teo}




\section{State equation}

\label{sec2}

\bigskip \setcounter{equation}{0}
In this section, we study the well-posedness of the state equation (\ref{NSy}%
) and deduce estimates for the state in terms of the control variables. Such
estimates will be fundamental to study the regularity (continuity,
differentiability) of the control-to-state mapping. Our strategy relies on
Galerkin's approximation method, by taking into account some useful results
on elliptic equations and compactness arguments.

Let us introduce the notion of solution to the system (\ref{NSy}), which
should be understood in the weak sense, according to the next definition.

\begin{definition}
\label{1def} The weak solution of the system \eqref{NSy}\ is a divergence
free function $\mathbf{y}\in L_{2}(0,T;H^{1}(\Omega )),$ \ satisfying the
boundary condition \
\begin{equation*}
\mathbf{y}\cdot \mathbf{n}=a\qquad \text{on}\quad \Gamma _{T}
\end{equation*}%
and being the solution of the integral equality%
\begin{align}
\int_{\Omega _{T}}\{-\mathbf{y}\cdot \partial _{t}\boldsymbol{\psi }+&
\left( \left( \mathbf{y\cdot }\nabla \right) \mathbf{y}\right) \cdot
\boldsymbol{\psi }+2\,D(\mathbf{y}):D(\boldsymbol{\psi })\,\}d\mathbf{x}dt
\notag \\
=& \int_{\Gamma _{T}}(b-\alpha (y\cdot {\bm{\tau }}))(\boldsymbol{\psi }%
\cdot {\bm{\tau }})\,d\mathbf{\gamma }\,dt+\int_{\Omega }\mathbf{y}_{0}\cdot
\boldsymbol{\psi }(0)\,d\mathbf{x}  \label{res1}
\end{align}%
for any $\boldsymbol{\psi }\in H^{1}(0,T;V)$ with \ $\boldsymbol{\psi }%
(T)=0. $
\end{definition}

\bigskip

The well-posedness of \ the system \eqref{NSy} will be presented at the end
of this section. Before we establish crucial intermediate results.

Let us introduce the function $\mathbf{a}=\nabla h_{{a}},$ where $h_{a}$ \
is the solution of the system%
\begin{equation}
\left\{
\begin{array}{l}
-\Delta h_{{a}}=0\quad \text{ in }\Omega , \\
\frac{\partial h_{{a}}}{\partial \mathbf{n}}=a\quad \text{ on }\Gamma%
\end{array}%
\right. \qquad \text{a.e. on\quad }(0,T).  \label{ha}
\end{equation}%
The function $\mathbf{a}$ \ satisfies Calderon-Zygmund%
\'{}%
s estimates%
\begin{align}
& ||\mathbf{a}||_{C(\overline{\Omega })}\leqslant C||\mathbf{a}%
||_{W_{p}^{1}(\Omega )}\leqslant C_{p}||a||_{W_{p}^{1-\frac{1}{p}}(\Gamma )},
\notag \\
& ||\partial _{t}\mathbf{a}||_{L_{2}(\Omega )}\leqslant C||\partial
_{t}a||_{W_{2}^{-\frac{1}{2}}(\Gamma )}\qquad \qquad \text{a.e. on\quad }%
(0,T).  \label{cal}
\end{align}%
where the constants $C_{p}$ depend on $2<p<\infty $ \ (see \cite{nec},
Theorem 9.9, p. 230 in \ \cite{cal} and Theorem 1.8, \ p. 12 \& Theorem
1.10, \ p. 15 in \cite{gir}). Accounting the regularity \eqref{eq00sec12}
and the embedding theorem $\ H^{1}(0,T)\hookrightarrow C([0,T])$ \ \ ( also
we refer to Lemma \ref{LM} ) \ we have that%
\begin{eqnarray}
\mathbf{a} &\in &L_{2}(0,T;C(\overline{\Omega })),\qquad \partial _{t}%
\mathbf{a}\in L_{2}(\Omega _{T}),  \notag \\
\mathbf{a} &\in &C([0,T];L_{2}(\Omega )).  \label{calderon}
\end{eqnarray}

\bigskip

The existence of solution for the system \eqref{NSy}\textbf{\ }will be shown
by Galerkin's method. There exists a sequence $\left\{ \mathbf{e}%
_{k}\right\} _{k=1}^{\infty }\subset H^{3}(\Omega ),$ being a basis for $V$
and an orthonormal basis for $H,$ which satisfies the Navier slip boundary
condition
\begin{equation}
\left[ 2D(\mathbf{e}_{k})\mathbf{n}+\alpha \mathbf{e}_{k}\right] \cdot {%
\bm{\tau }}=0  \label{nsbc}
\end{equation}%
on $\Gamma _{T}$ by Lemma 2.2. of \cite{clop} (see also Theorem 1 of \cite%
{S73}).\

For any fixed $n=1,2,....$ let $V_{n}=\mathrm{span}\,\{\mathbf{e}_{1},\ldots
,\mathbf{e}_{n}\}$\ and set $\mathbf{y}_{n}=\mathbf{u}_{n}+\mathbf{a}$ with%
\begin{equation*}
\mathbf{u}_{n}(t)=\sum_{k=1}^{n}c_{k}^{(n)}(t)\ \mathbf{e}_{k}
\end{equation*}%
being the solution of the integral equation%
\begin{eqnarray}
\int_{\Omega }\partial _{t}\mathbf{y}_{n}\cdot \boldsymbol{\psi }\ d\mathbf{x%
} &+&\int_{\Omega }\left\{ \left( \left( \mathbf{y}_{n}\cdot \nabla \right)
\mathbf{y}_{n}\right) \cdot \boldsymbol{\psi }+2\,D(\mathbf{y}_{n}):D(%
\boldsymbol{\psi })\right\} d\mathbf{x}  \notag \\
&=&\int_{\Gamma }(b-\alpha (\mathbf{y}_{n}\cdot {\bm{\tau })})(\boldsymbol{%
\psi }\cdot {\bm{\tau }})\,d\mathbf{\gamma }\,,\text{\qquad }\forall
\boldsymbol{\psi }\in V_{n},  \notag \\
\mathbf{u}_{n}(0) &=&\mathbf{u}_{n,0}.  \label{y1}
\end{eqnarray}%
Here $\mathbf{u}_{n,0}$ is the orthogonal projection of $\mathbf{u}_{0}(%
\mathbf{x})=\mathbf{y}_{0}(\mathbf{x})-\mathbf{a}(0,\mathbf{x})\in H$ \ onto
the space $V_{n}.$

In the following Proposition we will show the solvability of the system %
\eqref{y1}.

\begin{proposition}
\label{existence_state} Under the assumptions (\ref{eq00sec12}) the system %
\eqref{y1} has a solution $\mathbf{y}_{n}=\mathbf{u}_{n}+\mathbf{a}$, such
that%
\begin{align}
& \left\Vert \mathbf{u}_{n}\right\Vert _{L_{\infty }(0,T;L_{2}(\Omega
))}^{2}+\left\Vert D(\mathbf{u}_{n})\right\Vert _{L_{2}(\Omega _{T})}^{2}+||%
\sqrt{\alpha }\mathbf{u}_{n}||_{L_{2}(\Gamma _{T})}^{2}  \notag \\
& \leqslant C(\Vert \mathbf{u}_{n}(0)\Vert _{L_{2}(\Omega )}^{2}+||(a,b)||_{%
\mathcal{H}_{p}(0,T;\Gamma )}^{2}+1)\mathrm{\exp }(C||(a,b)||_{\mathcal{H}%
_{p}(0,T;\Gamma )}^{2})  \label{un}
\end{align}%
and
\begin{equation}
||\partial _{t}\mathbf{y}_{n}||_{L_{2}(0,T;H^{-1}(\Omega ))}^{2}\leqslant
C(\left\Vert \mathbf{y}_{0}\right\Vert _{L_{2}(\Omega )}^{2}+||(a,b)||_{%
\mathcal{H}_{p}(0,T;\Gamma )}^{2}+1).  \label{unt}
\end{equation}
\end{proposition}

\begin{proof}
The equation (\ref{y1}) defines a system of ordinary differential equations
in $\mathbb{R}^{2}$ with locally Lipschitz nonlinearities. Hence there
exists a local-in-time solution $\ \mathbf{u}_{n}$ in the space $%
C([0,T_{n}];V_{n})$. The global-in-time existence of $\mathbf{u}_{n}$
follows from a priori estimate \eqref{un}, which is valid for any $n=1,2,....
$\ \ Therefore we focus our attention on the deduction of the estimate %
\eqref{un}.

By firstly writing the equation \eqref{y1}$_{1}$ in terms of $\mathbf{u}_{n}$
and $\mathbf{a}$, taking $\boldsymbol{\psi }=\mathbf{e}_{k}$, multiplying by
$c_{k}^{(n)}$ and summing on $k=1,...,n,$ we derive
\begin{align}
\frac{1}{2}\frac{d}{dt}\int_{\Omega }|\mathbf{u}_{n}|^{2}\,d\mathbf{x}&
+2\int_{\Omega }|D(\mathbf{u}_{n})|^{2}\,d\mathbf{x}+\int_{\Gamma }\alpha (%
\mathbf{u}_{n}\cdot \bm{\tau })^{2}\,\,d\mathbf{\gamma }  \notag \\
& =\int_{\Gamma }\left\{ -\frac{a}{2}(\mathbf{u}_{n}\cdot \bm{\tau }%
)^{2}\,+\left( b-\alpha (\mathbf{a}\cdot \bm{\tau })\right) (\mathbf{u}%
_{n}\cdot \bm{\tau })\right\} \,d\mathbf{\gamma }  \notag \\
& -\int_{\Omega }\left[ \partial _{t}\mathbf{a}+\left( (\mathbf{u}_{n}+%
\mathbf{a})\cdot \nabla \right) \mathbf{a}\right] \cdot \mathbf{u}_{n}\,d%
\mathbf{x}  \notag \\
& -2\int_{\Omega }D(\mathbf{a}):D(\mathbf{u}_{n})\,d\mathbf{x}%
=I_{1}+I_{2}+I_{3}.  \label{es1}
\end{align}%
Considering the inequality \eqref{ab} for an appropriate $\varepsilon >0$
and the inequalities \ \eqref{LI}-\eqref{Korn} and \eqref{calderon}, \ the
terms $I_{1},$ $I_{2}$ and $I_{3}$ are estimated as follows
\begin{align*}
I_{1}& \leqslant (\Vert a\Vert _{L_{\infty }(\Gamma )}\,+1)\Vert \mathbf{u}%
_{n}\Vert _{L_{2}(\Gamma )}^{2}+\Vert b-\alpha (\mathbf{a}\cdot \bm{\tau }%
)\Vert _{L_{2}(\Gamma )}^{2}\, \\
& \leqslant C(||a||_{W_{p}^{1-\frac{1}{p}}(\Gamma )}+1)^{2}||\mathbf{u}%
_{n}||_{L_{2}(\Omega )}^{2}+\frac{1}{3}||D(\mathbf{u}_{n})||_{L_{2}(\Omega
)}^{2} \\
& +C(\Vert b\Vert _{L_{2}(\Gamma )}^{2}+\Vert \alpha \Vert _{L_{\infty
}(\Gamma )}\,||a||_{W_{p}^{1-\frac{1}{p}}(\Gamma )}^{2}\,),
\end{align*}%
\begin{eqnarray*}
I_{2} &\leqslant &\left( \Vert \partial _{t}\mathbf{a}\Vert _{L_{2}(\Omega
)}+\,||\mathbf{a}||_{C(\overline{\Omega })}\Vert \nabla \mathbf{a}\Vert
_{L_{2}(\Omega )}\right) \,||\mathbf{u}_{n}||_{L_{2}(\Omega )}+\,\Vert
\nabla \mathbf{a}\Vert _{L_{2}(\Omega )}\,\Vert \mathbf{u}_{n}\Vert
_{L_{4}(\Omega )}^{2} \\
&\leqslant &(||\partial _{t}a||_{W_{2}^{-\frac{1}{2}}(\Gamma
)}+\,||a||_{W_{p}^{1-\frac{1}{p}}(\Gamma )}^{2})\,||\mathbf{u}%
_{n}||_{L_{2}(\Omega )}\, \\
&&+C\,||a||_{W_{p}^{1-\frac{1}{p}}(\Gamma )}^{2}\,\Vert \mathbf{u}_{n}\Vert
_{L_{2}(\Omega )}^{2}+\frac{1}{3}||D(\mathbf{u}_{n})||_{L_{2}(\Omega )}^{2}
\end{eqnarray*}%
and
\begin{eqnarray*}
I_{3} &\leqslant &C\,\Vert D(\mathbf{a})\Vert _{L_{2}(\Omega )}^{2}\,+\frac{1%
}{3}||D(\mathbf{u}_{n})||_{L_{2}(\Omega )}^{2} \\
&\leqslant &C\,\,||a||_{W_{p}^{1-\frac{1}{p}}(\Gamma )}^{2}+\frac{1}{3}||D(%
\mathbf{u}_{n})||_{L_{2}(\Omega )}^{2}.
\end{eqnarray*}%
Combining the estimates of the terms $I_{1},$ $I_{2}$ and $I_{3}$ and %
\eqref{es1}, we obtain
\begin{eqnarray*}
\frac{1}{2}\frac{d}{dt}||\mathbf{u}_{n}\Vert _{L_{2}(\Omega )}^{2}
&+&\int_{\Omega }|D(\mathbf{u}_{n})|^{2}\,d\mathbf{x}+\int_{\Gamma }\alpha (%
\mathbf{u}_{n}\cdot \bm{\tau })^{2}\,\,d\mathbf{\gamma } \\
&\leqslant &h(t)(||\mathbf{u}_{n}\Vert _{L_{2}(\Omega )}^{2}+||\mathbf{u}%
_{n}\Vert _{L_{2}(\Omega )}+1)
\end{eqnarray*}%
with
\begin{equation*}
h(t)=C\left[ 1+\left( 1+\Vert \alpha \Vert _{L_{\infty }(\Gamma
)}^{2}\right) ||a||_{W_{p}^{1-\frac{1}{p}}(\Gamma )}^{2}+\,||\partial
_{t}a||_{W_{2}^{-\frac{1}{2}}(\Gamma )}^{2}+\Vert b\Vert _{L_{2}(\Gamma
)}^{2}\right]
\end{equation*}%
which belongs to $L_{1}(0,T)$ due to \eqref{cal} and \eqref{eq00sec12}. \
Applying Gronwall's inequality, we deduce \eqref{un}.

Now we show (\ref{unt}). The integration by parts gives
\begin{equation*}
\int_{\Omega }\left( \left( \mathbf{y}_{n}\cdot \nabla \right) \mathbf{y}%
_{n}\right) \cdot \boldsymbol{\psi }\,\,d\mathbf{x}=\int_{\Gamma }a\left(
\mathbf{y}_{n}\cdot \boldsymbol{\psi }\right) \,\,d\mathbf{\gamma }%
-\int_{\Omega }\left( \left( \mathbf{y}_{n}\cdot \nabla \right) \boldsymbol{%
\psi }\right) \cdot \mathbf{y}_{n}\,\,d\mathbf{x}.
\end{equation*}%
Therefore, the identity \eqref{y1} permit to deduce
\begin{align*}
|(\partial _{t}\mathbf{y}_{n},\boldsymbol{\psi })_{L_{2}(\Omega )}|&
\leqslant C\left( \Vert a\Vert _{L_{\infty }(\Gamma )}\Vert \mathbf{y}%
_{n}\Vert _{H^{1}(\Omega )}+\Vert \mathbf{y}_{n}\Vert _{L_{4}(\Omega
)}^{2}\right) \Vert \boldsymbol{\psi }\Vert _{H^{1}(\Omega )} \\
& +||D(\mathbf{y}_{n})||_{L_{2}(\Omega )}||D(\boldsymbol{\psi }%
)||_{L_{2}(\Omega )} \\
& +\left( ||b||_{L_{2}(\Gamma )}+||\sqrt{\alpha }\mathbf{y}%
_{n}||_{L_{2}(\Gamma )}\right) ||\boldsymbol{\psi }||_{L_{2}(\Gamma )}.
\end{align*}%
that gives
\begin{align*}
||\partial _{t}\mathbf{y}_{n}||_{H^{-1}(\Omega )}& =\sup_{\boldsymbol{\psi }%
\in H_{0}^{1}(\Omega )}\left\{ |(\partial _{t}\mathbf{y}_{n},\boldsymbol{%
\psi })_{L_{2}(\Omega )}|:\quad ||\boldsymbol{\psi }||_{H^{1}(\Omega
)}=1\right\} \\
& \leqslant C\bigl(\Vert a\Vert _{L_{_{\infty }}(\Gamma )}\Vert \mathbf{y}%
_{n}\Vert _{H^{1}(\Omega )}+\Vert \mathbf{y}_{n}\Vert _{L_{4}(\Omega )}^{2}
\\
& +||D(\mathbf{y}_{n})||_{L_{2}(\Omega )}+||b||_{L_{2}(\Gamma )}+||\sqrt{%
\alpha }\mathbf{y}_{n}||_{L_{2}(\Gamma )}\bigr).
\end{align*}%
Taking into account \eqref{LI} we have
\begin{align*}
\int_{0}^{T}\left( \Vert \mathbf{y}_{n}\Vert _{L_{4}(\Omega )}^{2}\right)
^{2}\;dt& \leqslant \int_{0}^{T}\left( ||\mathbf{y}_{n}||_{L_{2}(\Omega
)}^{1/2}||\nabla \mathbf{y}_{n}||_{L_{2}(\Omega )}^{1/2}+||\mathbf{y}%
_{n}||_{L_{2}(\Omega )}\right) ^{4}\;dt \\
& \leqslant C(\Vert \mathbf{y}_{n}\Vert _{L_{\infty }\left( 0,T;L_{2}(\Omega
)\right) }^{2}\Vert \mathbf{y}_{n}\Vert _{L_{2}\left( 0,T;H^{1}(\Omega
)\right) }^{2}+\Vert \mathbf{y}_{n}\Vert _{L_{\infty }\left(
0,T;L_{2}(\Omega )\right) }^{4}) \\
& \leqslant C,
\end{align*}%
that yields (\ref{unt}) by \eqref{cal}-\eqref{calderon} and \eqref{un}.
\end{proof}

\bigskip

\begin{teo}
\label{existence_state_y} Assume that the hypothesis (\ref{eq00sec12}) hold,
then the system \eqref{NSy} has a unique weak solution $\mathbf{y}$, such
that
\begin{equation}
\mathbf{y}\in C\left( [0,T];L_{2}(\Omega )\right) \cap L_{2}\left(
0,T;H^{1}(\Omega )\right) ,\quad \partial _{t}\mathbf{y}\in
L_{2}(0,T;H^{-1}(\Omega )).  \label{yyy}
\end{equation}%
Moreover, the following estimates hold%
\begin{align}
& \left\Vert \mathbf{y}\right\Vert _{C([0,T];L_{2}(\Omega ))}^{2}+\left\Vert
\mathbf{y}\right\Vert _{L_{2}(0,T;H^{1}(\Omega ))}^{2}+||\sqrt{\alpha }%
\mathbf{y}||_{L_{2}(\Gamma _{T})}^{2}  \notag \\
& \leqslant C(\left\Vert \mathbf{y}_{0}\right\Vert _{L_{2}(\Omega
)}^{2}+||(a,b)||_{\mathcal{H}_{p}(0,T;\Gamma )}^{2}+1)\mathrm{\exp }%
(C||(a,b)||_{\mathcal{H}_{p}(0,T;\Gamma )}^{2}),  \label{uny}
\end{align}%
\begin{equation}
||\partial _{t}\mathbf{y}||_{L_{2}(0,T;H^{-1}(\Omega ))}^{2}\leqslant
C(\left\Vert \mathbf{y}_{0}\right\Vert _{L_{2}(\Omega )}^{2}+||(a,b)||_{%
\mathcal{H}_{p}(0,T;\Gamma )}^{2}+1).  \label{unty}
\end{equation}
\end{teo}

\begin{proof}
The estimates \eqref{cal}, \eqref{calderon}, (\ref{un}) and (\ref{unt})
imply that the sequence of the functions
\begin{equation*}
\mathbf{u}_{n}\in L_{2}\left( 0,T;V\right) ,\qquad \partial _{t}\mathbf{u}%
_{n}\in L_{2}(0,T;H^{-1}(\Omega )),
\end{equation*}%
are uniformly bounded, for $n=1,2,....$, so, we can apply the compactness
argument of \cite{sim} and take a suitable subsequence of $\left\{ \mathbf{u}%
_{n}\right\} ,$ such that
\begin{eqnarray*}
\mathbf{y}_{n} &=&\mathbf{u}_{n}+\mathbf{a}\rightharpoonup \mathbf{y}=%
\mathbf{u}+\mathbf{a}\quad \mbox{ weakly in
}\ L_{\infty }\left( 0,T;L_{2}(\Omega )\right) \cap L_{2}\left(
0,T;H^{1}(\Omega )\right) ,\qquad \\
\partial _{t}\mathbf{y}_{n} &\rightharpoonup &\partial _{t}\mathbf{y}\qquad
\mbox{ weakly in
}\ L_{2}(0,T;H^{-1}(\Omega )), \\
\mathbf{y}_{n} &\rightarrow &\mathbf{y}\qquad \mbox{ strongly in
}\ L_{2}(\Omega _{T}).
\end{eqnarray*}%
Hence integrating over the time interval $(0,T)$ and passing to the limit as
$n\rightarrow \infty $ in \eqref{y1}, we deduce that the function $\mathbf{y}%
=\mathbf{u}+\mathbf{a}$ is a weak solution of \eqref{NSy} in the sense of
the definition \ref{1def}.

The properties $\mathbf{y}\in L_{2}\left( 0,T;H^{1}(\Omega )\right) $, $%
\partial _{t}\mathbf{y}\in L_{2}(0,T;H^{-1}(\Omega ))$ and Lemma \ref{LM}
yield
\begin{equation*}
\mathbf{y}\in C([0,T];L_{2}(\Omega )),
\end{equation*}%
which gives a meaning for the initial condition for $\mathbf{y}$ in (\ref%
{NSy}). Finally, accounting (\ref{cal})-(\ref{calderon}), we derive (\ref%
{uny})-(\ref{unty}).

The uniqueness result is a direct consequence of Proposition \ref{Lips},
that we will show in the following section.
\end{proof}

\bigskip

\section{Lipschitz continuity of the control-to-state mapping}

\label{sec3}\setcounter{equation}{0} This section is devoted to the study of
the Lipschitz continuity to the state $\mathbf{y}$ as a function of the
control variables $a, b$. This regularity result will be necessary in
Section 7 in order to analyse the G\^ateaux differentiability of this
function. 

\begin{proposition}
\label{Lips} Let $\left( \mathbf{y}_{1},p_{1}\right) $ and $\left( \mathbf{y}%
_{2},p_{2}\right) $ \ be two weak solutions for the system \eqref{NSy} with
two corresponding boundary conditions $a_{1},\;b_{1}$ and $a_{2},$ $b_{2},$
but with the same initial condition $\mathbf{y}_{0}.$ Denoting by $\widehat{%
\mathbf{y}}=\mathbf{y}_{1}-\mathbf{y}_{2}$, we have
\begin{equation}
\left\Vert \widehat{\mathbf{y}}\right\Vert _{C([0,T];L_{2}(\Omega
))}^{2}+\left\Vert D\left( \widehat{\mathbf{y}}\right) \right\Vert
_{L_{2}(\Omega _{T})}^{2}\mathbf{+}||\sqrt{\alpha }\widehat{\mathbf{y}}%
||_{L_{2}(\Gamma _{T})}^{2}\leqslant C||(\widehat{a},\widehat{b})||_{%
\mathcal{H}_{p}(0,T;\Gamma )}^{2}  \label{dif00}
\end{equation}%
with $\widehat{b}=b_{1}-b_{2}$ and $\widehat{a}=a_{1}-a_{2}$.
\end{proposition}

\begin{proof}
Let us denote $\widehat{\mathbf{a}}=\nabla h_{\widehat{a}},$ where $h_{%
\widehat{a}}$ is the solution of the system (\ref{ha}) with $a=\widehat{a}$.
\ \

We easily verify that the functions $\mathbf{w}=\widehat{\mathbf{y}}-%
\widehat{\mathbf{a}},$ $\widehat{p}=p_{1}-p_{2}$ satisfy the system%
\begin{equation}
\left\{
\begin{array}{l}
\partial _{t}\mathbf{w}+\left( \mathbf{y}_{2}\cdot \nabla \right) \mathbf{w}%
-\bigtriangledown \widehat{p}=\Delta \mathbf{w}+\mathbf{F},\quad \mathrm{div}%
\,\mathbf{w}=0\;\quad \text{ in}\quad \Omega _{T}, \\
\\
\mathbf{w}\cdot \mathbf{n}=\mathbf{0},\;\quad \left[ 2D(\mathbf{w})\,\mathbf{%
n}+\alpha \mathbf{w}\right] \cdot {\bm{\tau }}=\tilde{b}\;\;\quad \text{on}%
\quad \Gamma _{T}, \\
\\
\mathbf{w}(0,\mathbf{x})=-\widehat{\mathbf{a}}(0,\mathbf{x})\;\quad \text{in}%
\quad \Omega%
\end{array}%
\right.  \label{NSV}
\end{equation}%
with $\mathbf{F}=-\partial _{t}\widehat{\mathbf{a}}+\Delta \widehat{\mathbf{a%
}}-\left( \left( \mathbf{w}+\widehat{\mathbf{a}}\right) \cdot \nabla \right)
\mathbf{y}_{1}-\left( \mathbf{y}_{2}\cdot \nabla \right) \widehat{\mathbf{a}}
$ \ and $\tilde{b}=\widehat{b}-\left[ 2D(\widehat{\mathbf{a}})\,\mathbf{n}%
+\alpha \widehat{\mathbf{a}}\right] \cdot ${\textbf{$\bm{\tau }$}}$.$

Therefore multiplying the first equation in \eqref{NSV} by $\mathbf{w}$ and
integrating over $\Omega ,$ we obtain%
\begin{align}
\frac{1}{2}\frac{d}{dt}\int_{\Omega }|\mathbf{w}|^{2}\,d\mathbf{x}&
+2\int_{\Omega }|D(\mathbf{w})|^{2}\,d\mathbf{x+}\int_{\Gamma }\alpha (%
\mathbf{w}\cdot \bm{\tau })^{2}\,d\mathbf{\gamma }  \notag \\
& =\int_{\Gamma }\left\{ -\frac{a_{2}}{2}(\mathbf{w}\cdot \bm{\tau }%
)^{2}\,+\left( \widehat{b}-\alpha (\widehat{\mathbf{a}}\cdot \bm{\tau }%
)\right) (\mathbf{w}\cdot \bm{\tau })\right\} \,d\mathbf{\gamma }  \notag \\
& -\int_{\Omega }\left[ \partial _{t}\widehat{\mathbf{a}}+\left( \left(
\mathbf{w}+\widehat{\mathbf{a}}\right) \cdot \nabla \right) \mathbf{y}_{1}%
\right] \cdot \mathbf{w}\,d\mathbf{x}  \notag \\
& -\int_{\Omega }\left[ \left( \mathbf{y}_{2}\cdot \nabla \right) \widehat{%
\mathbf{a}}\right] \cdot \mathbf{w}\,d\mathbf{x}-\int_{\Omega }2D(\widehat{%
\mathbf{a}}):D(\mathbf{w})\,d\mathbf{x}  \notag \\
& =J_{1}+J_{2}+J_{3}+J_{4}.  \label{eq2}
\end{align}%
Let us estimate the term $J_{1}.$ \ By \eqref{TT}, \eqref{eq00sec12}, %
\eqref{calderon} \ \ and the embedding $W_{p}^{1-\frac{1}{p}}(\Gamma
)\hookrightarrow L_{\infty }(\Gamma ),$ we deduce
\begin{align*}
J_{1}& \leqslant (\Vert a_{2}\Vert _{L_{\infty }(\Gamma )}\,+1)\Vert \mathbf{%
w}\Vert _{L_{2}(\Gamma )}^{2}+C\left( \Vert \widehat{b}\Vert _{L_{2}(\Gamma
)}^{2}+\Vert \alpha \Vert _{L_{\infty }(\Gamma )}^{2}\Vert (\widehat{\mathbf{%
a}}\cdot \bm{\tau })\Vert _{L_{2}(\Gamma )}^{2}\,\right) \\
& \leqslant f_{1}(t)||\mathbf{w}||_{L_{2}(\Omega )}^{2}+\frac{1}{4}||D(%
\mathbf{w})||_{L_{2}(\Omega )}^{2}+C\left( \Vert \widehat{b}\Vert
_{L_{2}(\Gamma )}^{2}+||\widehat{a}||_{W_{p}^{1-\frac{1}{p}}(\Gamma
)}^{2}\,\right)
\end{align*}%
with $f_{1}(t)=C(\Vert a_{2}\Vert _{W_{p}^{1-\frac{1}{p}}(\Gamma
)}\,+1)^{2}\in L_{1}(0,T)$ by \eqref{eq00sec12}. The term $J_{2}$ is
estimated as follows
\begin{eqnarray*}
J_{2} &\leqslant &\left( \Vert \partial _{t}\widehat{\mathbf{a}}\Vert
_{L_{2}(\Omega )}+\,\Vert \widehat{\mathbf{a}}\Vert _{C(\overline{\Omega }%
)}\Vert \nabla \mathbf{y}_{1}\Vert _{L_{2}(\Omega )}\right) \Vert \mathbf{w}%
\Vert _{L_{2}(\Omega )}+\,\Vert \nabla \mathbf{y}_{1}\Vert _{L_{2}(\Omega
)}\,\Vert \mathbf{w}\Vert _{L_{4}(\Omega )}^{2} \\
&\leqslant &\left( \Vert \partial _{t}\widehat{\mathbf{a}}\Vert
_{L_{2}(\Omega )}+\,\Vert \widehat{\mathbf{a}}\Vert _{C(\overline{\Omega }%
)}\right) \sqrt{f_{2}(t)}\Vert \mathbf{w}\Vert _{L_{2}(\Omega )} \\
&&+\Vert \nabla \mathbf{y}_{1}\Vert _{L_{2}(\Omega )}\Vert \mathbf{w}\Vert
_{L_{2}(\Omega )}\Vert \nabla \mathbf{w}\Vert _{L_{2}(\Omega )} \\
&\leqslant &f_{2}(t)\Vert \mathbf{w}\Vert _{L_{2}(\Omega )}^{2}+C\left(
||\partial _{t}\widehat{a}||_{W_{2}^{-\frac{1}{2}}(\Gamma )}^{2}+||\widehat{a%
}||_{W_{p}^{1-\frac{1}{p}}(\Gamma )}^{2}\right) +\frac{1}{4}||D(\mathbf{w}%
)||_{L_{2}(\Omega )}^{2}
\end{eqnarray*}%
with $f_{2}(t)=C(1+\Vert \nabla \mathbf{y}_{1}\Vert _{L_{2}(\Omega
)})^{2}\in L_{1}(0,T)$ by \eqref{yyy}. \ \ Using \eqref{LI} for $\mathbf{v}=%
\mathbf{y}_{2}$\ and \eqref{LI} for $\ \mathbf{v}=\mathbf{w}$, we have%
\begin{eqnarray*}
J_{3} &\leqslant &\,\Vert \mathbf{y}_{2}\Vert _{L_{4}(\Omega )}||\nabla
\widehat{\mathbf{a}}\Vert _{L_{2}(\Omega )}\Vert \mathbf{w}\Vert
_{L_{4}(\Omega )}\leqslant C||\nabla \widehat{\mathbf{a}}\Vert
_{L_{2}(\Omega )}\Vert \mathbf{y}_{2}\Vert _{L_{4}(\Omega )}||\mathbf{w}%
||_{L_{2}(\Omega )}^{1/2}||\nabla \mathbf{w}||_{L_{2}(\Omega )}^{1/2} \\
&\leqslant &C||\nabla \widehat{\mathbf{a}}\Vert _{L_{2}(\Omega )}^{2}+\Vert
\mathbf{y}_{2}\Vert _{L_{4}(\Omega )}^{2}||\mathbf{w}||_{L_{2}(\Omega
)}||\nabla \mathbf{w}||_{L_{2}(\Omega )} \\
&\leqslant &f_{3}(t)\Vert \mathbf{w}\Vert _{L_{2}(\Omega )}^{2}+C||\widehat{a%
}||_{W_{p}^{1-\frac{1}{p}}(\Gamma )}^{2}+\frac{1}{4}||D(\mathbf{w}%
)||_{L_{2}(\Omega )}^{2}
\end{eqnarray*}%
with\ $f_{3}(t)=C\Vert \mathbf{y}_{2}\Vert _{L_{4}(\Omega )}^{4}\leqslant
C\left( ||\mathbf{y}_{2}||_{L_{2}(\Omega )}^{1/2}||\nabla \mathbf{y}%
_{2}||_{L_{2}(\Omega )}^{1/2}+||\mathbf{y}_{2}||_{L_{2}(\Omega )}\right)
^{4}\in L_{1}(0,T)$ by \eqref{LI} and \eqref{yyy}.\ Finally we have%
\begin{eqnarray*}
J_{4} &\leqslant &C\,\Vert D(\widehat{\mathbf{a}})\Vert _{L_{2}(\Omega
)}^{2}\,+\frac{1}{4}||D(\mathbf{w})||_{L_{2}(\Omega )}^{2} \\
&\leqslant &C\,\,||\widehat{a}||_{W_{p}^{1-\frac{1}{p}}(\Gamma )}^{2}+\frac{1%
}{4}||D(\mathbf{u})||_{L_{2}(\Omega )}^{2}.
\end{eqnarray*}%
Combining the above deduced estimates of the terms $J_{1},$ $J_{2}$, $J_{3},$
\ $J_{4}$\ \ and \eqref{eq2}, we obtain
\begin{align}
\frac{d}{dt}||\mathbf{w}\Vert _{L_{2}(\Omega )}^{2}+\int_{\Omega }|D(\mathbf{%
w})|^{2}\,d\mathbf{x}& +\int_{\Gamma }\alpha (\mathbf{w}\cdot \bm{\tau }%
)^{2}\,d\mathbf{\gamma }\leqslant f(t)||\mathbf{w}\Vert _{L_{2}(\Omega )}^{2}
\notag \\
& +C\left\{ \Vert \widehat{b}\Vert _{L_{2}(\Gamma )}^{2}+\,||\partial _{t}%
\widehat{a}||_{W_{2}^{-\frac{1}{2}}(\Gamma )}^{2}+||\widehat{a}||_{W_{p}^{1-%
\frac{1}{p}}(\Gamma )}^{2}\right\}  \notag
\end{align}%
with $f(t)=f_{1}(t)+f_{2}(t)+f_{3}(t)\in L_{1}(0,T)$. Applying Gronwall's
inequality, we deduce
\begin{align}
& \left\Vert \mathbf{w}\right\Vert _{L_{\infty }(0,T;L_{2}(\Omega
))}^{2}+\left\Vert D\left( \mathbf{w}\right) \right\Vert _{L_{2}(\Omega
_{T})}^{2}\mathbf{+}||\sqrt{\alpha }\mathbf{w}||_{L_{2}(\Gamma _{T})}^{2}
\notag \\
& \leqslant C\biggl\{\left\Vert \widehat{\mathbf{a}}(0,\mathbf{x}%
)\right\Vert _{L_{2}(\Omega )}^{2}+\int_{0}^{T}(\Vert \widehat{b}\Vert
_{L_{2}(\Gamma )}^{2}+\,||\partial _{t}\widehat{a}||_{W_{2}^{-\frac{1}{2}%
}(\Gamma )}^{2}+||\widehat{a}||_{W_{p}^{1-\frac{1}{p}}(\Gamma )}^{2})dt%
\biggr\}.  \label{dif00w}
\end{align}

Therefore, taking into account that $\widehat{\mathbf{y}}=\mathbf{w}+%
\widehat{\mathbf{a}}$ and (\ref{cal})-(\ref{calderon}), we derive %
\eqref{dif00}.
\end{proof}

\bigskip

$\hfill $

\section{Linearized state equation}

\label{sec4}\setcounter{equation}{0}

This section deals with the well-posedness of the linearized state equation.
Let us mention that the existence and uniqueness of the linearized state is
of main importance to analyse the G\^{a}teaux derivative of the
control-to-state mapping. Moreover, its regularity plays a key roll in the
deduction of the duality property, relating the linearized state with the
adjoint state. We recall that such duality relation allows to write the
first order derivative of the cost functional in terms of the adjoint state,
yielding the so-called first order optimality condition.

Let us consider the solution $\mathbf{y}$ of the state system (\ref{NSy}),
then the corresponding linearized system reads as follows
\begin{equation}
\left\{
\begin{array}{ll}
\partial _{t}\mathbf{z}+(\mathbf{z}\cdot \nabla )\mathbf{y}+(\mathbf{y}\cdot
\nabla )\mathbf{z}+\nabla \pi =\Delta \mathbf{z,\quad }div\ \mathbf{z}=0, & %
\mbox{in}\ \Omega _{T},\vspace{2mm} \\
\mathbf{z}\cdot \mathbf{n}=f,\quad \left[ 2D(\mathbf{z})\,\mathbf{n}+\alpha
\mathbf{z}\right] \cdot \bm{\tau }=g & \quad \mbox{on}\ \Gamma _{T},\vspace{%
2mm} \\
\mathbf{z}(0)=0 & \quad \mbox{in}\ \Omega%
\end{array}%
\right.  \label{linearized}
\end{equation}%
with the boundary data%
\begin{equation}
(f,g)\in \mathcal{H}_{p}(0,T;\Gamma )\quad \quad \text{with\textit{\ \ }}%
p\in (2,+\infty )\text{ \ \ as in (\ref{eq00sec12}).}  \label{regf}
\end{equation}

Let us define $\mathbf{f}=\nabla h_{f}$ with $h_{f}$ being the solution of
the system (\ref{ha}). Then the function $\mathbf{f}$ \ satisfies the
estimates%
\begin{align}
& ||\mathbf{f}||_{C(\overline{\Omega })}\leqslant C||\mathbf{f}%
||_{W_{p}^{1}(\Omega )}\leqslant C||f||_{W_{p}^{1-\frac{1}{p}}(\Gamma )},
\notag \\
& ||\partial _{t}\mathbf{f}||_{L_{2}(\Omega )}\leqslant C||\partial
_{t}f||_{W_{2}^{-\frac{1}{2}}(\Gamma )}\qquad \qquad \text{a.e. on}\quad
(0,T),  \label{ref}
\end{align}%
and%
\begin{eqnarray}
\mathbf{f} &\in &L_{2}(0,T;C(\overline{\Omega })),\qquad \partial _{t}%
\mathbf{f}\in L_{2}(\Omega _{T}),  \notag \\
\mathbf{f} &\in &C([0,T];L_{2}(\Omega )).  \label{ref2}
\end{eqnarray}

\begin{definition}
The weak solution of the system (\ref{linearized})\ is the divergence free
function $\mathbf{z}\in L_{2}(0,T;H^{1}(\Omega ))$\ satisfying the boundary
condition \
\begin{equation*}
\mathbf{z}\cdot \mathbf{n}=f\qquad \text{on}\quad \Gamma _{T}
\end{equation*}%
and being the solution of the integral equality
\begin{eqnarray*}
\int_{\Omega _{T}}\{-\mathbf{z}\cdot \partial _{t}\boldsymbol{\psi } &+&%
\left[ (\mathbf{z}\cdot \nabla )\mathbf{y}+(\mathbf{y}\cdot \nabla )\mathbf{z%
}\right] \cdot \boldsymbol{\psi }+2\,D(\mathbf{z}):D(\boldsymbol{\psi })\,\}d%
\mathbf{x}dt \\
&=&\int_{\Gamma _{T}}(g-\alpha (\mathbf{z}\cdot {\bm{\tau })})(\boldsymbol{%
\psi }\cdot {\bm{\tau })}\,d\mathbf{\gamma }\,dt,
\end{eqnarray*}%
which is valid for all $\boldsymbol{\psi }\in H^{1}(0,T;V)$: $\boldsymbol{%
\psi }(T)=0.$
\end{definition}

\bigskip

In what follows we will establish the solvability of the system (\ref%
{linearized})

\begin{proposition}
\label{ex_uniq_lin} Under the assumptions \eqref{regf} there exists a unique
weak solution $\mathbf{z}$ for the system (\ref{linearized}), such that
\begin{eqnarray*}
\mathbf{z} &\in &C([0,T];L_{2}(\Omega ))\cap L_{2}\left( 0,T;H^{1}(\Omega
)\right) , \\
\partial _{t}\mathbf{z} &\in &L_{2}(0,T;H^{-1}(\Omega ))
\end{eqnarray*}%
and%
\begin{equation}
\left\Vert \mathbf{z}\right\Vert _{C([0,T];L_{2}(\Omega ))}^{2}+\left\Vert
D\left( \mathbf{z}\right) \right\Vert _{L_{2}(\Omega _{T})}^{2}\mathbf{+}||%
\sqrt{\alpha }\mathbf{z}||_{L_{2}(\Gamma _{T})}^{2}\leqslant C||(f,g)||_{%
\mathcal{H}_{p}(0,T;\Gamma )}^{2}.  \label{eqsa}
\end{equation}
\end{proposition}

\begin{proof}
Let us consider as in the Section \ref{sec2} the subspace $V_{n}=\mathrm{span%
}\,\{\mathbf{e}_{1},\ldots ,\mathbf{e}_{n}\}$ of $V$ and the sequence $%
\left\{ \mathbf{e}_{k}\right\} _{k=1}^{\infty }\subset H^{3}(\Omega )$ being
the orthogonal basis for $V$ and the orthonormal basis for $H,$ satisfying
the Navier slip boundary condition (\ref{nsbc}).

For any fixed $n=1,2,....$\ we define $\mathbf{z}_{n}=\tilde{\mathbf{z}}_{n}+%
\mathbf{f}$, where
\begin{equation*}
\tilde{\mathbf{z}}_{n}(t)=\sum_{k=1}^{n}r_{k}^{(n)}(t)\ \mathbf{e}_{k}
\end{equation*}%
is the solution for the differential equation
\begin{eqnarray}
\int_{\Omega }\partial _{t}\mathbf{z}_{n}\cdot \boldsymbol{\psi }\ d\mathbf{x%
} &+&\int_{\Omega }\left\{ \left[ (\mathbf{z}_{n}\cdot \nabla )\mathbf{y}+(%
\mathbf{y}\cdot \nabla )\mathbf{z}_{n}\right] \cdot \boldsymbol{\psi }+2\,D(%
\mathbf{z}_{n}):D(\boldsymbol{\psi })\,\right\} d\mathbf{x}  \notag \\
&=&\int_{\Gamma }(g-\alpha (\mathbf{z}_{n}\cdot {\bm{\tau }}))(\boldsymbol{%
\psi }\cdot {\bm{\tau }})\,d\mathbf{\gamma }\,,\text{\qquad }\forall
\boldsymbol{\psi }\in V_{n},  \notag \\
\mathbf{z}_{n}(0) &=&\mathbf{z}_{n,0}.  \label{z}
\end{eqnarray}%
Here $\mathbf{z}_{n,0}$ is the orthogonal projections in $H$ of $\tilde{%
\mathbf{z}}_{0}(\mathbf{x})=\mathbf{z}_{0}(\mathbf{x})-\mathbf{f}(0,\mathbf{x%
})$ \ onto the space $V_{n}.$ Since the equation (\ref{z}) is a system of
linear ordinary differential equations in $\mathbb{R}^{2},$ there exists a
global-in-time solution $\ \tilde{\mathbf{z}}_{n}$ in the space $%
C([0,T];V_{n})$.\

Let us show the validity of (\ref{eqsa}) for $\mathbf{z}=\mathbf{z}_{n}.$ If
we write the equation (\ref{z}) in terms of $\tilde{\mathbf{z}}_{n}$ and
choose the test function $\boldsymbol{\psi }=\tilde{\mathbf{z}}_{n}$, we
deduce%
\begin{align}
\frac{1}{2}\frac{d}{dt}\int_{\Omega }|\tilde{\mathbf{z}}_{n}|^{2}\,d\mathbf{x%
}& +2\int_{\Omega }|D(\tilde{\mathbf{z}}_{n})|^{2}\,d\mathbf{x+}\int_{\Gamma
}\alpha (\tilde{\mathbf{z}}_{n}\cdot \bm{\tau })^{2}\,d\mathbf{\gamma }
\notag \\
& =\int_{\Gamma }\left\{ -\frac{a}{2}(\tilde{\mathbf{z}}_{n}\cdot \bm{\tau }%
)^{2}\,+\left( g-\alpha (\mathbf{f}\cdot \bm{\tau })\right) (\tilde{\mathbf{z%
}}_{n}\cdot \bm{\tau })\right\} \,d\mathbf{\gamma }  \notag \\
& -\int_{\Omega }\left[ \partial _{t}\mathbf{f}+\left( \left( \tilde{\mathbf{%
z}}_{n}+\mathbf{f}\right) \cdot \nabla \right) \mathbf{y}\right] \cdot
\tilde{\mathbf{z}}_{n}\,d\mathbf{x}  \notag \\
& \,-\int_{\Omega }\left[ \left( \mathbf{y}\cdot \nabla \right) \mathbf{f}%
\right] \cdot \tilde{\mathbf{z}}_{n}\,d\mathbf{x-}\int_{\Omega }2D(\mathbf{f}%
):D(\tilde{\mathbf{z}}_{n})\,d\mathbf{x}  \notag \\
& =J_{1}+J_{2}+J_{3}+J_{4}.  \label{eqeq2}
\end{align}%
Let us estimate the terms $J_{1},J_{2}$ and $J_{3}$. We have%
\begin{align*}
J_{1}& \leqslant (\Vert f\Vert _{L_{\infty }(\Gamma )}\,+1)\Vert \tilde{%
\mathbf{z}}_{n}\Vert _{L_{2}(\Gamma )}^{2}+C\left( \Vert g\Vert
_{L_{2}(\Gamma )}^{2}+\Vert \alpha \Vert _{L_{\infty }(\Gamma )}^{2}\Vert (%
\mathbf{f}\cdot \bm{\tau })\Vert _{L_{2}(\Gamma )}^{2}\,\right) \\
& \leqslant C(\Vert f\Vert _{W_{p}^{1-\frac{1}{p}}(\Gamma )}+1)||\tilde{%
\mathbf{z}}_{n}||_{L_{2}(\Omega )}||\nabla \tilde{\mathbf{z}}%
_{n}||_{L_{2}(\Omega )}+C\left( \Vert g\Vert _{L_{2}(\Gamma
)}^{2}+||f||_{W_{p}^{1-\frac{1}{p}}(\Gamma )}^{2}\,\right) \\
& \leqslant h_{1}(t)||\tilde{\mathbf{z}}_{n}||_{L_{2}(\Omega )}^{2}+\frac{1}{%
4}||D(\tilde{\mathbf{z}}_{n})||_{L_{2}(\Omega )}^{2}+C\left( \Vert g\Vert
_{L_{2}(\Gamma )}^{2}+||f||_{W_{p}^{1-\frac{1}{p}}(\Gamma )}^{2}\,\right)
\end{align*}%
with $h_{1}(t)=C(\Vert a\Vert _{W_{p}^{1-\frac{1}{p}}(\Gamma )}+1)^{2}\in
L_{1}(0,T)$ by \eqref{eq00sec12}.%
\begin{eqnarray*}
J_{2} &\leqslant &\left( \Vert \partial _{t}\mathbf{f}\Vert _{L_{2}(\Omega
)}+\,\Vert \mathbf{f}\Vert _{C(\overline{\Omega })}\Vert \nabla \mathbf{y}%
\Vert _{L_{2}(\Omega )}\right) \Vert \tilde{\mathbf{z}}_{n}\Vert
_{L_{2}(\Omega )} \\
&&+\,\Vert \nabla \mathbf{y}\Vert _{L_{2}(\Omega )}\,\Vert \tilde{\mathbf{z}}%
_{n}\Vert _{L_{4}(\Omega )}^{2}\leqslant \left( \Vert \partial _{t}\mathbf{f}%
\Vert _{L_{2}(\Omega )}+\,\Vert \mathbf{f}\Vert _{C(\overline{\Omega }%
)}\right) \sqrt{h_{2}(t)}\Vert \tilde{\mathbf{z}}_{n}\Vert _{L_{2}(\Omega )}
\\
&&+\Vert \nabla \mathbf{y}\Vert _{L_{2}(\Omega )}\Vert \tilde{\mathbf{z}}%
_{n}\Vert _{L_{2}(\Omega )}\Vert \nabla \tilde{\mathbf{z}}_{n}\Vert
_{L_{2}(\Omega )} \\
&\leqslant &h_{2}(t)\Vert \tilde{\mathbf{z}}_{n}\Vert _{L_{2}(\Omega )}^{2}+%
\frac{1}{4}||D(\tilde{\mathbf{z}}_{n})||_{L_{2}(\Omega )}^{2}+\left(
||\partial _{t}f||_{W_{2}^{-\frac{1}{2}}(\Gamma )}^{2}+||f||_{W_{p}^{1-\frac{%
1}{p}}(\Gamma )}^{2}\right)
\end{eqnarray*}%
with $h_{2}(t)=C(1+\Vert \nabla \mathbf{y}\Vert _{L_{2}(\Omega )})^{2}\in
L_{1}(0,T)$ by \eqref{uny}. \ $\ $ Reasoning as in Proposition \ref{Lips} we
derive%
\begin{eqnarray*}
J_{3} &\leqslant &\,\Vert \mathbf{y}\Vert _{L^{4}(\Omega )}||\nabla \mathbf{f%
}\Vert _{L_{2}(\Omega )}\Vert \tilde{\mathbf{z}}_{n}\Vert _{L_{4}(\Omega
)}\leqslant ||\nabla \mathbf{f}\Vert _{L_{2}(\Omega )}|\Vert \mathbf{y}\Vert
_{L^{4}(\Omega )}|\tilde{\mathbf{z}}_{n}||_{L_{2}(\Omega )}^{1/2}||\nabla
\tilde{\mathbf{z}}_{n}||_{L_{2}(\Omega )}^{1/2} \\
&\leqslant &||\nabla \mathbf{f}\Vert _{L_{2}(\Omega )}^{2}+\Vert \mathbf{y}%
\Vert _{L^{4}(\Omega )}^{2}||\tilde{\mathbf{z}}_{n}||_{L_{2}(\Omega
)}||\nabla \tilde{\mathbf{z}}_{n}||_{L_{2}(\Omega )} \\
&\leqslant &||f||_{W_{p}^{1-\frac{1}{p}}(\Gamma )}^{2}+h_{3}(t)\Vert \tilde{%
\mathbf{z}}_{n}\Vert _{L_{2}(\Omega )}^{2}+\frac{1}{4}||D(\tilde{\mathbf{z}}%
_{n})||_{L_{2}(\Omega )}^{2}
\end{eqnarray*}%
with $h_{3}(t)=C\Vert \mathbf{y}\Vert _{L^{4}(\Omega )}^{4}\leqslant C\left(
||\mathbf{y}||_{L_{2}(\Omega )}^{1/2}||\nabla \mathbf{y}||_{L_{2}(\Omega
)}^{1/2}+||\mathbf{y}||_{L_{2}(\Omega )}\right) ^{4}\in L_{1}(0,T)$ \ by %
\eqref{LI} and \eqref{uny}. The last term $J_{4}$ is estimated \ as%
\begin{eqnarray*}
J_{4} &\leqslant &C\,\Vert D(\mathbf{f})\Vert _{L_{2}(\Omega )}^{2}\,+\frac{1%
}{4}||D(\tilde{\mathbf{z}}_{n})||_{L_{2}(\Omega )}^{2} \\
&\leqslant &C\,\,||f||_{W_{p}^{1-\frac{1}{p}}(\Gamma )}^{2}+\frac{1}{4}||D(%
\tilde{\mathbf{z}}_{n})||_{L_{2}(\Omega )}^{2}.
\end{eqnarray*}

Therefore the above deduced estimates of the terms $J_{1},$ $J_{2}$, $J_{3}$%
, $J_{3}$ \ and \eqref{eqeq2} imply the inequality
\begin{align}
\frac{1}{2}\frac{d}{dt}||\tilde{\mathbf{z}}_{n}\Vert _{L_{2}(\Omega
)}^{2}+\int_{\Omega }|D(\tilde{\mathbf{z}}_{n})|^{2}\,d\mathbf{x}&
+\int_{\Gamma }\alpha (\tilde{\mathbf{z}}_{n}\cdot \bm{\tau })^{2}\,d\mathbf{%
\gamma }\leqslant h(t)||\tilde{\mathbf{z}}_{n}\Vert _{L_{2}(\Omega )}^{2}+
\notag \\
& +C\left\{ \Vert g\Vert _{L_{2}(\Gamma )}^{2}+\,||\partial _{t}f||_{W_{2}^{-%
\frac{1}{2}}(\Gamma )}^{2}+||f||_{W_{p}^{1-\frac{1}{p}}(\Gamma
)}^{2}\right\}   \notag
\end{align}%
with $h(t)=h_{1}(t)+h_{2}(t)+h_{3}(t)\in L_{1}(0,T)$. Hence Gronwall's
inequality gives
\begin{eqnarray}
&&\left\Vert \tilde{\mathbf{z}}_{n}\right\Vert _{L_{\infty
}(0,T;L_{2}(\Omega ))}^{2}+\left\Vert D\left( \tilde{\mathbf{z}}_{n}\right)
\right\Vert _{L_{2}(\Omega _{T})}^{2}\mathbf{+}||\sqrt{\alpha }\tilde{%
\mathbf{z}}_{n}||_{L_{2}(\Gamma _{T})}^{2}\leqslant C\biggl\{\left\Vert
\mathbf{f}(0,\mathbf{x})\right\Vert _{L_{2}(\Omega )}^{2}  \notag \\
&&+\int_{0}^{T}\left\{ \Vert g\Vert _{L_{2}(\Gamma )}^{2}+\,||\partial
_{t}f||_{W_{2}^{-\frac{1}{2}}(\Gamma )}^{2}+||f||_{W_{p}^{1-\frac{1}{p}%
}(\Gamma )}^{2}\right\} \ dt\biggr\}.  \label{zf}
\end{eqnarray}%
This estimate and \eqref{z} permit to obtain that the sequence
\begin{equation*}
\partial _{t}\tilde{\mathbf{z}}_{n}\in L_{2}(0,T;H^{-1}(\Omega ))
\end{equation*}%
is uniformly bounded on $n=1,2,....$ Hence using the compactness argument of
\cite{sim}, there exists a suitable subsequence of $\left\{ \tilde{\mathbf{z}%
}_{n}\right\} ,$ such that
\begin{eqnarray}
\tilde{\mathbf{z}}_{n} &\rightarrow &\tilde{\mathbf{z}}\qquad
\mbox{ weakly in
}\ L_{2}(0,T;H^{1}(\Omega )),\qquad   \notag \\
\partial _{t}\tilde{\mathbf{z}}_{n} &\rightarrow &\partial _{t}\tilde{%
\mathbf{z}}\qquad \mbox{ weakly in
}\ L_{2}(0,T;H^{-1}(\Omega )),  \notag \\
\tilde{\mathbf{z}}_{n} &\rightarrow &\tilde{\mathbf{z}}\qquad
\mbox{ strongly in
}\ L_{2}(\Omega _{T}).  \label{conv_z}
\end{eqnarray}%
\ Passing on $n\rightarrow \infty $ in \eqref{z}, we deduce that
\begin{equation*}
\tilde{\mathbf{z}}\in L_{\infty }(0,T;L_{2}(\Omega ))\cap
L_{2}(0,T;H^{1}(\Omega )),\qquad \partial _{t}\tilde{\mathbf{z}}\in
L_{2}(0,T;H^{-1}(\Omega )).
\end{equation*}%
Hence $\mathbf{z}=\tilde{\mathbf{z}}+\mathbf{f}$ is the weak solution of (%
\ref{linearized}), which satisfies \eqref{eqsa} by Lemma \ref{LM}, \eqref{zf}
and \eqref{regf}-\eqref{ref2}.\ \ The uniqueness result follows from the
linearity of the system by taking into account the estimates \eqref{eqsa}.
\end{proof}

\bigskip

\section{G\^{a}teaux differentiability of the control-to-state mapping}

\label{sec5}\setcounter{equation}{0}

To deduce the necessary first-order optimality conditions, we should study
the the G\^{a}teaux differentiability of the cost functional $J$, which
requires the determination of the G\^{a}teaux derivative of the
control-to-state mapping. The goal of this section is to show that the G\^{a}%
teaux derivative of the control-to-state mapping $(a,b)\rightarrow \mathbf{y}
$, at a point $(a,b)$, in any direction $(f,g)$, exists and is given by the
solution of the linearized system (\ref{linearized}).

\begin{proposition}
\label{Gat} For given $(a,b)$ and $\mathbf{y}_{0}$ satisfying (\ref%
{eq00sec12}) and
\begin{equation*}
(f,g)\in \mathcal{H}_{p}(0,T;\Gamma ),
\end{equation*}%
let us consider
\begin{equation*}
a_{\varepsilon }=a+\varepsilon f,\quad b_{\varepsilon }=b+\varepsilon
g\qquad \forall \varepsilon \in (0,1).
\end{equation*}%
If $\left( \mathbf{y},\pi \right) $ and $\left( \mathbf{y}_{\varepsilon
},\pi _{\varepsilon }\right) $ are the solutions of (\ref{NSy})
corresponding to $(a,b,\mathbf{y}_{0})$ and $(a_{\varepsilon
},b_{\varepsilon },\mathbf{y}_{0}),$ respectively, then the following
representation holds
\begin{equation}
\mathbf{y}_{\varepsilon }=\mathbf{y}+\varepsilon \mathbf{z}+\varepsilon \,%
\boldsymbol{\delta }_{\varepsilon }\quad \mbox{   with  }\quad
\lim_{\varepsilon \rightarrow 0}\sup_{t\in {[0,T]}}\left\Vert \boldsymbol{%
\delta }_{\varepsilon }\right\Vert _{L_{2}(\Omega )}^{2}=0,  \label{gateau_1}
\end{equation}%
where
\begin{equation*}
\mathbf{z}\in C([0,T];H)\cap L_{2}(0,T;V)\quad
\end{equation*}%
is the solution of (\ref{linearized}) satisfying the estimates (\ref{eqsa}).
\end{proposition}

\begin{proof}
It is straightforward to verify that $\mathbf{z}_{\varepsilon }=\frac{%
\mathbf{y}_{\varepsilon }-\mathbf{y}}{\varepsilon }$ and $\tilde{\pi}%
_{\varepsilon }=\frac{\pi _{\varepsilon }-\pi }{\varepsilon }$ satisfy the
system
\begin{equation}
\left\{
\begin{array}{l}
\partial _{t}\mathbf{z}_{\varepsilon }+\left( \mathbf{y}\cdot \nabla \right)
\mathbf{z}_{\varepsilon }+\left( \mathbf{z}_{\varepsilon }\cdot \nabla
\right) \mathbf{y}_{\varepsilon }-\nabla \tilde{\pi}_{\varepsilon }=\Delta
\mathbf{z}_{\varepsilon },\quad \mathrm{div}\,\mathbf{z}_{\varepsilon
}=0\;\quad \text{ in}\quad \Omega _{T}, \\
\\
\mathbf{z}_{\varepsilon }\cdot \mathbf{n}=f,\;\quad \left[ 2D(\mathbf{z}%
_{\varepsilon })\,\mathbf{n}+\alpha \mathbf{z}_{\varepsilon }\right] \cdot {%
\ \mathbf{\bm{\tau }}}=g\;\;\quad \text{on}\quad \Gamma _{T}, \\
\\
\mathbf{z}_{\varepsilon }(0,\mathbf{x})=0\;\quad \text{in}\quad \Omega
\end{array}%
\right.   \label{NSVG}
\end{equation}%
and $\boldsymbol{\delta }_{\varepsilon }=\mathbf{z}_{\varepsilon }-\mathbf{z}
$ fulffills the system
\begin{equation}
\left\{
\begin{array}{l}
\partial _{t}\boldsymbol{\delta }_{\varepsilon }+\left( \mathbf{y}\cdot
\nabla \right) \boldsymbol{\delta }_{\varepsilon }+\left( \boldsymbol{\delta
}_{\varepsilon }\cdot \nabla \right) \mathbf{y}_{\varepsilon }+\left(
\mathbf{z}\cdot \nabla \right) \left( \mathbf{y}_{\varepsilon }-\mathbf{y}%
\right)  \\
\qquad -\nabla \left( \tilde{\pi}_{\varepsilon }-\hat{\pi}\right) =\Delta
\boldsymbol{\delta }_{\varepsilon },\qquad \mathrm{div}\,\boldsymbol{\delta }%
_{\varepsilon }=0\;\quad \text{ in}\quad \Omega _{T}, \\
\\
\boldsymbol{\delta }_{\varepsilon }\cdot \mathbf{n}=0,\;\quad \left[ 2D(%
\boldsymbol{\delta }_{\varepsilon })\,\mathbf{n}+\alpha \boldsymbol{\delta }%
_{\varepsilon }\right] \cdot {\mathbf{\bm{\tau }}}=0\;\;\quad \text{on}\quad
\Gamma _{T}, \\
\\
\boldsymbol{\delta }_{\varepsilon }(0,\mathbf{x})=0\;\quad \text{in}\quad
\Omega .%
\end{array}%
\right.   \label{NS00}
\end{equation}%
Multiplying the first equation of the last system by $\boldsymbol{\delta }%
_{\varepsilon }$ and integrating over $\Omega $, we deduce
\begin{align}
\frac{1}{2}\frac{d}{dt}\int_{\Omega }|\boldsymbol{\delta }_{\varepsilon
}|^{2}\,d\mathbf{x}& +2\int_{\Omega }|D(\boldsymbol{\delta }_{\varepsilon
})|^{2}\,d\mathbf{x}+\int_{\Gamma }\alpha (\boldsymbol{\delta }_{\varepsilon
}\cdot \bm{\tau })^{2}\,\,d\mathbf{\gamma }=-\int_{\Gamma }\frac{a}{2}(%
\boldsymbol{\delta }_{\varepsilon }\cdot \bm{\tau })^{2}\,d\mathbf{\gamma }
\notag \\
& -\int_{\Omega }\left[ \left( \boldsymbol{\delta }_{\varepsilon }\cdot
\nabla \right) \mathbf{y}_{\varepsilon }\right] \cdot \boldsymbol{\delta }%
_{\varepsilon }\,d\mathbf{x}-\int_{\Omega }\left[ \left( \mathbf{z}\cdot
\nabla \right) \left( \mathbf{y}_{\varepsilon }-\mathbf{y}\right) \right]
\cdot \boldsymbol{\delta }_{\varepsilon }\,d\mathbf{x}  \notag \\
& =I_{1}+I_{2}+I_{3}.  \label{er}
\end{align}%
Applying the inequalities \eqref{ab}, \eqref{LI}-\eqref{Korn} and %
\eqref{calderon}, the following estimates hold
\begin{equation*}
I_{1}\leqslant C\Vert a\Vert _{L_{\infty }(\Gamma )}\Vert \boldsymbol{\delta
}_{\varepsilon }\Vert _{L_{2}(\Gamma )}^{2}\leqslant C||a||_{W_{p}^{1-\frac{1%
}{p}}(\Gamma )}^{2}||\boldsymbol{\delta }_{\varepsilon }||_{L_{2}(\Omega
)}^{2}+\frac{1}{3}||D(\boldsymbol{\delta }_{\varepsilon })||_{L_{2}(\Omega
)}^{2},
\end{equation*}%
\begin{equation*}
I_{2}\leqslant C\Vert \mathbf{y}_{\varepsilon }\Vert _{H^{1}(\Omega )}\Vert
\boldsymbol{\delta }_{\varepsilon }\Vert _{L_{4}(\Omega )}^{2}\leqslant
C\Vert \mathbf{y}_{\varepsilon }\Vert _{H^{1}(\Omega )}^{2}\Vert \boldsymbol{%
\delta }_{\varepsilon }\Vert _{L_{2}(\Omega )}^{2}+\frac{1}{3}||D(%
\boldsymbol{\delta }_{\varepsilon })||_{L_{2}(\Omega )}^{2}
\end{equation*}%
and
\begin{align*}
I_{3}& \leqslant C\Vert \mathbf{y}_{\varepsilon }-\mathbf{y}\Vert
_{H^{1}(\Omega )}\Vert \mathbf{z}\Vert _{L_{4}(\Omega )}\Vert \boldsymbol{%
\delta }_{\varepsilon }\Vert _{L_{4}(\Omega )} \\
& \leqslant C\Vert \mathbf{y}_{\varepsilon }-\mathbf{y}\Vert _{H^{1}(\Omega
)}^{2}+C\Vert \mathbf{z}\Vert _{L_{4}(\Omega )}^{4}\Vert \boldsymbol{\delta }%
_{\varepsilon }\Vert _{L_{2}(\Omega )}^{2}+\frac{1}{3}||D(\boldsymbol{\delta
}_{\varepsilon })||_{L_{2}(\Omega )}^{2}.
\end{align*}%
Then we obtain
\begin{align}
\frac{1}{2}\frac{d}{dt}\int_{\Omega }|\boldsymbol{\delta }_{\varepsilon
}|^{2}\,d\mathbf{x}& +\int_{\Omega }|D(\boldsymbol{\delta }_{\varepsilon
})|^{2}\,d\mathbf{x}+\int_{\Gamma }\alpha (\boldsymbol{\delta }_{\varepsilon
}\cdot \bm{\tau })^{2}\,\,d\mathbf{\gamma }  \notag \\
& \leqslant Cf(t)||\boldsymbol{\delta }_{\varepsilon }||_{L_{2}(\Omega
)}^{2}+C\Vert \mathbf{y}_{\varepsilon }-\mathbf{y}\Vert _{H^{1}(\Omega )}^{2}
\notag
\end{align}%
with\ $f(t)=(||a||_{W_{p}^{1-\frac{1}{p}}(\Gamma )}^{2}+\Vert \mathbf{y}%
_{\varepsilon }\Vert _{H^{1}(\Omega )}^{2}+\Vert \mathbf{z}\Vert
_{L_{4}(\Omega )}^{4})\in L_{1}(0,T)$ \ by \eqref{eq00sec12}, \eqref{uny} and%
\begin{equation}
\Vert \mathbf{z}\Vert _{L^{4}(\Omega )}^{4}\leqslant \left( ||\mathbf{z}%
||_{L_{2}(\Omega )}^{1/2}||\nabla \mathbf{z}||_{L_{2}(\Omega )}^{1/2}+||%
\mathbf{z}||_{L_{2}(\Omega )}\right) ^{4}\in L_{1}(0,T)  \label{ra}
\end{equation}%
by \eqref{LI}, \eqref{eqsa}.

Applying Gronwall's inequality and using \eqref{eq00sec12}, we deduce
\begin{align}
\Vert \boldsymbol{\delta }_{\varepsilon }\Vert _{L_{\infty
}(0,T;L_{2}(\Omega ))}^{2}& +\Vert D(\boldsymbol{\delta }_{\varepsilon
})\Vert _{L_{2}(\Omega _{T})}^{2}+\Vert \sqrt{\alpha }\boldsymbol{\delta }%
_{\varepsilon }\Vert _{L_{2}(\Gamma _{T})}^{2}\leqslant C\varepsilon
^{2}\Vert \mathbf{y}_{\varepsilon }-\mathbf{y}\Vert _{L_{2}(0,T;H^{1}(\Omega
)}^{2}  \notag \\
& \leqslant C\varepsilon ^{2}||(f,g)||_{\mathcal{H}_{p}(0,T;\Gamma
)}^{2}\rightarrow 0\text{\qquad as }\varepsilon \rightarrow 0,  \label{errr}
\end{align}%
according to (\ref{dif00}) and (\ref{cal}). On the other hand, using the
same reasoning as for the state and linearized equation and the above
estimates, we can also deduce that
\begin{equation*}
|(\partial _{t}\boldsymbol{\delta }_{\varepsilon },\boldsymbol{\psi }%
)_{L_{2}(\Omega )}|\leqslant \beta (t)||\boldsymbol{\psi }||_{H^{1}(\Omega
)},
\end{equation*}%
with $\beta (t)\in L_{2}(0,T)$, which gives
\begin{equation}
||\partial _{t}\boldsymbol{\delta }_{\varepsilon
}||_{L_{2}(0,T;H^{-1}(\Omega ))}<\infty .  \label{dtG}
\end{equation}%
Finally, (\ref{errr}) and (\ref{dtG}) yield (\ref{gateau_1}).
\end{proof}

\bigskip

\bigskip

As a direct consequence of Proposition \ref{Gat}, we easily derive the
following result on the variation for the cost functional (\ref{cost}).

\begin{proposition}
\label{Gat1} Assume that $(a,b)$, $(f,g)$, $\mathbf{y}_{0}$, $\mathbf{z}$ \
\ and
\begin{equation*}
a_{\varepsilon }=a+\varepsilon f,\quad b_{\varepsilon }=b+\varepsilon
g,\qquad \forall \varepsilon \in (0,1)
\end{equation*}%
satisfy the assumptions of Proposition \ref{Gat}. Then we have
\begin{equation*}
J\left( a_{\varepsilon },b_{\varepsilon },\mathbf{y}_{\varepsilon }\right)
=J\left( a,b,\mathbf{y}\right) +\varepsilon \left\{ \int_{\Omega _{T}}\left(
\mathbf{y}-\mathbf{y}_{d}\right) \cdot \mathbf{z}\,d\mathbf{x}%
dt+\int_{\Gamma _{T}}\left( \lambda _{1}af+\lambda _{2}bg\right) \,d\mathbf{%
\gamma }dt\right\} +o(\varepsilon ),
\end{equation*}%
where $\mathbf{y}$, $\mathbf{y}_{\varepsilon }$ are the solutions of (\ref%
{NSy}), corresponding to $(a,b,\mathbf{y}_{0})$, $(a_{\varepsilon
},b_{\varepsilon },\mathbf{y}_{0})$ and $\mathbf{z}$ is the solution of (\ref%
{linearized}).
\end{proposition}

\bigskip

\section{Adjoint equation}

\label{sec6}\setcounter{equation}{0}

This section is devoted to the study of the adjoint system. The existence
and uniqueness of the solution is shown by the same approach that we have
considered to study the state and linearized state equations. Namely, we
will use Galerkin's approximations and compactness arguments.

Let $\mathbf{y}$ be the solution of the state equation \eqref{NSy}
corresponding to the given data $(a,b,\mathbf{y}_{0})$. The adjoint system
is given by
\begin{equation}
\left\{
\begin{array}{ll}
-\partial _{t}\mathbf{p}-2D(\mathbf{p})\mathbf{y}+\nabla \pi =\Delta \mathbf{%
p}+\mathbf{U,\qquad }\mathrm{div}\,\mathbf{p}=0 & \quad \mbox{in}\ \Omega
_{T},\vspace{2mm} \\
\mathbf{p}\cdot \mathbf{n}=0\qquad \left[ 2D(\mathbf{p})\mathbf{n}+(\alpha
+a)\mathbf{p}\right] \cdot \bm{\tau }=0 & \quad \mbox{on}\ \Gamma _{T},%
\vspace{2mm} \\
\mathbf{p}(T)=0 & \quad \mbox{in}\ \Omega .%
\end{array}%
\right.   \label{adjoint}
\end{equation}%
%
%
%
%
%
%
%
%
%
%
%
%
%
%
%
%
%
%
%
%
%
%
%

\begin{definition}
A function $\mathbf{p}\in L_{2}(0,T;V)$ is a weak solution of (\ref{adjoint}%
) if the integral equality%
\begin{align}
& \int_{\Omega _{T}}\left\{ \mathbf{p}\cdot \partial _{t}\boldsymbol{\phi }%
-\left( 2D(\mathbf{p})\mathbf{y}\right) \cdot \boldsymbol{\phi }+2D(\mathbf{p%
}):\,D(\boldsymbol{\phi })-\mathbf{U}\cdot \boldsymbol{\phi }\right\} \,d%
\mathbf{x}dt\vspace{2mm}  \notag \\
& =-\int_{\Gamma _{T}}(\alpha +a)(\mathbf{p}\cdot \bm{\tau })(\boldsymbol{%
\phi }\cdot \bm{\tau })\,d\mathbf{\gamma }dt  \label{p}
\end{align}%
is valid for all $\boldsymbol{\phi }\in H^{1}(0,T;V)$: \ $\boldsymbol{\phi }%
(0)=0.$
\end{definition}

\begin{proposition}
\label{ex_uniq_adj} Assume that $\mathbf{U}\in L_{2}(\Omega _{T}).$ Under
the assumptions (\ref{eq00sec12}) there exists a unique weak solution ($%
\mathbf{p,}\pi )$ \ for the system (\ref{adjoint}), such that
\begin{equation*}
\mathbf{p}\in C([0,T];H)\cap L_{2}(0,T;V),\qquad \pi \in
H^{-1}(0,T;L_{2}(\Omega )).
\end{equation*}%
Moreover, the following estimate holds
\begin{equation}
\left\Vert \mathbf{p}\right\Vert _{C([0,T];L_{2}(\Omega ))}^{2}+\left\Vert
D\left( \mathbf{p}\right) \right\Vert _{L_{2}(\Omega _{T})}^{2}\mathbf{+}||%
\sqrt{\alpha }\mathbf{p}||_{L_{2}(\Gamma _{T})}^{2}\leqslant C||\mathbf{U}%
||_{L_{2}(\Omega _{T})}^{2}.  \label{estp}
\end{equation}
\end{proposition}

\begin{proof}
First, let us notice that according to p. 49-50 of \cite{kel2} there exists
a sequence $\left\{ \widetilde{\mathbf{e}}_{k}\right\} _{k=1}^{\infty
}\subset H^{3}(\Omega ),$ being a basis for $V$ \ and an orthonormal basis
for $H,$ of eigenfunctions of the Stokes problem%
\begin{equation}
\left\{
\begin{array}{ll}
-\Delta \widetilde{\mathbf{e}}_{k}+\nabla \widetilde{\pi }_{k}=\lambda _{k}%
\widetilde{\mathbf{e}}_{k},\mathbf{\qquad }\mbox{div}\,\widetilde{\mathbf{e}}%
_{k}=0, & \mbox{in}{\ \Omega },\vspace{2mm} \\
\widetilde{\mathbf{e}}_{k}\cdot \mathbf{n}=0,\;\quad \left[ 2D(\widetilde{%
\mathbf{e}}_{k})\,\mathbf{n}+(a+\alpha )\widetilde{\mathbf{e}}_{k}\right]
\cdot {\bm{\tau }}=0\;\quad & \text{on}\ \Gamma.%
\end{array}%
\right.  \label{y4}
\end{equation}%
For a more detailed description, we refer to a similar situation described
in \cite{evans}, p. 297-307: Theorem 2, p. 300 and Theorem 5, p. 305 (see
also Definition 1-4 and Theorem 1-16, p. 63 of \cite{aub}).

The existence of solution for the system \eqref{adjoint}\textbf{\ }will be
shown by Galerkin's method. \ For any fixed $n=1,2,....,$ as in Proposition %
\ref{existence_state}, we consider the subspace $\widetilde{V}_{n}=\mathrm{%
span}\,\{\widetilde{\mathbf{e}}_{1},\ldots ,\widetilde{\mathbf{e}}_{n}\}$ of
$V$\ and define
\begin{equation}  \label{GA}
\mathbf{p}_{n}(t)=\sum_{j=1}^{n}s_{j}^{(n)}(t)\ \widetilde{\mathbf{e}}_{j}
\end{equation}
as the solution of the equation%
\begin{eqnarray}
&&\int_{\Omega }\left\{ -\partial _{t}\mathbf{p}_{n}\cdot \boldsymbol{\psi }%
\ -\left( 2D(\mathbf{p}_{n})\mathbf{y}\right) \cdot \boldsymbol{\psi }+2\,D(%
\mathbf{p}_{n}):D(\boldsymbol{\psi })-\mathbf{U}\cdot \boldsymbol{\psi }%
\right\} \,d\mathbf{x}dt\vspace{2mm}\,  \notag \\
&=&-\int_{\Gamma }\left\{ (a+\alpha )(\mathbf{p}_{n}\cdot \bm{\tau })(%
\boldsymbol{\psi }\cdot \bm{\tau })\right\} \,d\mathbf{\gamma },\text{\qquad
}\forall \boldsymbol{\psi }\in V_{n},  \notag \\
&&\mathbf{p}_{n}(T)=\mathbf{0}.  \label{pn}
\end{eqnarray}%
Since the equation (\ref{pn}) is a system of linear ordinary differential
equations in $\mathbb{R}^{n},$ there exists a global-in-time solution $\
\mathbf{p}_{n}$ in the space $C([0,T];V_{n})$.

Now, we show the estimate (\ref{estp}) for $\mathbf{p}=\mathbf{p}_{n}.$
Taking $\boldsymbol{\psi }=\widetilde{\mathbf{e}}_{j}$ in (\ref{pn}),
multiplying it by $s_{j}^{(n)}$ and summing on $j=1,...,n,$ we verify that (%
\ref{pn}) holds for $\boldsymbol{\psi }=\mathbf{p}_{n}$ yielding
\begin{align}
-\frac{1}{2}\frac{d}{dt}\int_{\Omega }|\mathbf{p}_{n}|^{2}\,d\mathbf{x}&
+2\int_{\Omega }|D(\mathbf{p}_{n})|^{2}\,d\mathbf{x+}\int_{\Gamma }\alpha (%
\mathbf{p}_{n}\cdot \bm{\tau })^{2}\,d\mathbf{\gamma }  \notag \\
& =-\int_{\Gamma }\frac{a}{2}|\mathbf{p}_{n}|^{2}\,d\mathbf{\gamma }%
+\int_{\Omega }\left[ \left( \nabla ^{T}\mathbf{p}_{n}\right) \mathbf{y}+%
\mathbf{U}\right] \cdot \mathbf{p}_{n}\,d\mathbf{x}=J_{1}+J_{2}.
\label{eqeq22}
\end{align}%
\

Let us estimate the terms $J_{1}$ and $J_{2}$. We have%
\begin{align*}
J_{1}& \leqslant C\Vert a\Vert _{L_{\infty }(\Gamma )}\Vert \mathbf{p}%
_{n}\Vert _{L_{2}(\Gamma )}^{2}\leqslant C\Vert a\Vert _{W_{p}^{1-\frac{1}{p}%
}(\Gamma )}||\mathbf{p}_{n}||_{L_{2}(\Omega )}||\nabla \mathbf{p}%
_{n}||_{L_{2}(\Omega )} \\
& \leqslant h_{1}(t)||\mathbf{p}_{n}||_{L_{2}(\Omega )}^{2}+\frac{1}{2}||D(%
\mathbf{p}_{n})||_{L_{2}(\Omega )}^{2}
\end{align*}%
with $h_{1}(t)=C\Vert a\Vert {}_{W_{p}^{1-\frac{1}{p}}(\Gamma )}^{2}\in
L_{1}(0,T)$ by \eqref{eq00sec12}. Applying the Gagliardo-Nirenberg-Sobolev
inequality (\ref{LI}) with $q=4$ and Young's inequality (\ref{yi}), we obtain%
\begin{eqnarray*}
J_{2} &\leqslant &\,\Vert \nabla \mathbf{p}_{n}\Vert _{L_{2}(\Omega
)}\,\Vert \mathbf{p}_{n}\Vert _{L_{4}(\Omega )}\Vert \mathbf{y}\Vert
_{L_{4}(\Omega )}+\Vert \mathbf{U}\Vert _{L_{2}(\Omega )}\Vert \mathbf{p}%
_{n}\Vert _{L_{2}(\Omega )} \\
&\leqslant &\Vert \mathbf{y}\Vert _{L_{2}(\Omega )}^{1/2}\Vert \nabla
\mathbf{y}\Vert _{L_{2}(\Omega )}^{1/2}\Vert \mathbf{p}_{n}\Vert
_{L_{2}(\Omega )}^{1/2}\Vert \nabla \mathbf{p}_{n}\Vert _{L_{2}(\Omega
)}^{3/2}+\Vert \mathbf{U}\Vert _{L_{2}(\Omega )}\Vert \mathbf{p}_{n}\Vert
_{L_{2}(\Omega )} \\
&\leqslant &h_{2}(t)\Vert \mathbf{p}_{n}\Vert _{L_{2}(\Omega )}^{2}+\frac{1}{%
2}||D(\mathbf{p}_{n})||_{L_{2}(\Omega )}^{2}+\Vert \mathbf{U}\Vert
_{L_{2}(\Omega )}^{2}
\end{eqnarray*}%
with $h_{2}(t)=C(1+\Vert \mathbf{y}\Vert _{L_{2}(\Omega )}^{2}\Vert \nabla
\mathbf{y}\Vert _{L_{2}(\Omega )}^{2})\in L_{1}(0,T)$ by \eqref{uny}. \
Therefore the above deduced estimates of the terms $J_{1},$ $J_{2}$ \ and %
\eqref{eqeq22} imply
\begin{equation}
-\frac{1}{2}\frac{d}{dt}||\mathbf{p}_{n}\Vert _{L_{2}(\Omega
)}^{2}+2\int_{\Omega }|D(\mathbf{p}_{n})|^{2}\,d\mathbf{x}+\int_{\Gamma
}\alpha (\mathbf{p}_{n}\cdot \bm{\tau })^{2}\,d\mathbf{\gamma }\leqslant
h(t)||\mathbf{p}_{n}\Vert _{L_{2}(\Omega )}^{2}+\Vert \mathbf{U}\Vert
_{L_{2}(\Omega )}^{2}  \notag
\end{equation}%
with $h(t)=h_{1}(t)+h_{2}(t)\in L_{1}(0,T)$ depending only on the data %
\eqref{eq00sec12} of our problem \eqref{NSy}.\ Hence integrating the
obtained inequality over the time interval $(t,T),$ we derive Gronwall's
inequality, which gives
\begin{equation}
\left\Vert \mathbf{p}_{n}\right\Vert _{L_{\infty }(0,T;L_{2}(\Omega
))}^{2}+\left\Vert D\left( \mathbf{p}_{n}\right) \right\Vert _{L_{2}(\Omega
_{T})}^{2}\mathbf{+}||\sqrt{\alpha }\mathbf{p}_{n}||_{L_{2}(\Gamma
_{T})}^{2}\leqslant C\int_{0}^{T}\Vert \mathbf{U}\Vert _{L_{2}(\Omega
)}^{2}\ dt.  \label{p_n}
\end{equation}%
This estimate and \eqref{pn} permit to conclude that the sequence%
\begin{equation*}
\partial _{t}\mathbf{p}_{n}\in L_{2}(0,T;H^{-1}(\Omega ))
\end{equation*}%
is uniformly bounded on $n=1,2,....$ \ which allows to use the compactness
argument of \cite{sim}. Therefore for a suitable subsequence of $\left\{
\mathbf{p}_{n}\right\} ,$ we have that
\begin{eqnarray}
\mathbf{p}_{n} &\rightarrow &\mathbf{p}\qquad \mbox{ weakly in
}\ L_{2}(0,T;H^{1}(\Omega )),\qquad   \notag \\
\partial _{t}\mathbf{p}_{n} &\rightarrow &\partial _{t}\mathbf{p}_{n}\qquad
\mbox{ weakly in
}\ L_{2}(0,T;H^{-1}(\Omega )),  \notag \\
\mathbf{p}_{n} &\rightarrow &\mathbf{p}\qquad \mbox{ strongly in
}\ L_{2}(\Omega _{T}).  \label{conv_p}
\end{eqnarray}%
Taking the limit on $n\rightarrow \infty $ in \eqref{pn}, we derive that
\begin{equation*}
\mathbf{p}\in L_{2}(0,T;H^{1}(\Omega )),\qquad \partial _{t}\mathbf{p}\in
L_{2}(0,T;H^{-1}(\Omega ))
\end{equation*}%
is the weak solution of (\ref{adjoint}), satisfying \eqref{estp}. By the
result given on the page 208 of \cite{tem}, we deduce the existence of the
pressure $\pi \in H^{-1}(0,T;L_{2}(\Omega )).$

The uniqueness follows from the linearity of the system and the estimates %
\eqref{estp}.
\end{proof}

\bigskip

In the next section, we will prove that the adjoint state $\mathbf{p}$ and
the linearized state $\mathbf{z}$ are related through a suitable integration
by parts formula. In order to give a meaning to certain boundary terms that
will appear in that duality relation, it is necessary to improve the
regularity properties of the adjoint state.

\begin{proposition}
\label{reg_extra} Under the assumptions of Proposition \ref{ex_uniq_adj} and
the additional regularity for the data
\begin{equation*}
a,\alpha \in H^{1}(0,T;L_{\infty }(\Gamma )),
\end{equation*}
the pair
\begin{equation}
\mathbf{p}\in C([0,T];L_{2}(\Omega ))\cap L_{2}(0,T;H^{2}(\Omega )),\qquad
\pi \in L_{2}(0,T;H^{1}(\Omega ))  \label{up}
\end{equation}%
satisfies the system (\ref{adjoint}) in the usual sense.
\end{proposition}

\begin{proof}
Let us consider Galerkin's approximations $\mathbf{p}_{n}$ defined in (\ref%
{GA})-(\ref{pn}). Since the unction $\mathbf{p}_{n}(t,\cdot )\in
H^{3}(\Omega )\cap V$ \ fulfills Navier's boundary condition (see (\ref%
{adjoint})), then integrating by parts the equality (\ref{pn}), we obtain
\begin{eqnarray}
\int_{\Omega }\left( \partial _{t}\mathbf{p}_{n}+2D(\mathbf{p}_{n})\mathbf{y}%
+\triangle \mathbf{p}_{n}+\mathbf{U}\right) \cdot \boldsymbol{\psi \ }d%
\mathbf{x} &=&0,\text{\qquad }\forall \boldsymbol{\psi }\in V_{n},  \notag \\
\mathbf{p}_{n}(T) &=&\mathbf{0}.  \label{help}
\end{eqnarray}%
Let us introduce the Helmholtz projector $\mathbb{P}_{n}:L_{2}(\Omega
)\longrightarrow \widetilde{V}_{n}$ of $V$ and define the function $A\mathbf{%
p}_{n}=\mathbb{P}_{n}\left( -\triangle \mathbf{p}_{n}\right) =-\triangle
\mathbf{p}_{n}+\nabla \widehat{\pi }_{n}\in \widetilde{V}_{n}$ for some $%
\widehat{\pi }_{n}\in H^{1}(\Omega ).$

Taking $\boldsymbol{\psi }=\widetilde{\mathbf{e}}_{j}$ in (\ref{help}),
multiplying it by $\lambda _{j}s_{j}^{(n)}$ and summing on $j=1,...,n,$ we
verify that (\ref{help}) is valid for the test function $\boldsymbol{\psi }=A%
\mathbf{p}_{n}$, that implies the following equality%
\begin{align}
-\int_{\Omega }\partial _{t}\mathbf{p}_{n}\cdot A\mathbf{p}_{n}\,d\mathbf{x}%
& +\int_{\Omega }|A\mathbf{p}_{n}|^{2}\,d\mathbf{x}  \notag \\
& =\int_{\Omega }\mathbf{U\cdot }A\mathbf{p}_{n}d\mathbf{x+}\int_{\Omega
}\left( 2D(\mathbf{p}_{n})\mathbf{y}\right) \cdot A\mathbf{p}_{n}d\mathbf{x}.
\notag
\end{align}%
Applying (\ref{integrate}) and accounting $\mathbf{p}_{n}\cdot \mathbf{n}=0$
on $\Gamma ,$ \ we have%
\begin{eqnarray*}
\int_{\Omega }\partial _{t}\mathbf{p}_{n}\cdot A\mathbf{p}_{n}\ d\mathbf{x}
&=&-\int_{\Gamma }\left( 2D(\mathbf{p}_{n})\mathbf{n}\right) \cdot \partial
_{t}\mathbf{p}_{n}\ d\mathbf{\gamma }+\int_{\Omega }2\,D(\mathbf{p}%
_{n}):D(\partial _{t}\mathbf{p}_{n})\,d\mathbf{x} \\
&=&-\int_{\Gamma }\left[ \left( 2D(\mathbf{p}_{n})\mathbf{n}\right) \cdot %
\bm{\tau }\right] \partial _{t}(\mathbf{p}_{n}\cdot \bm{\tau })\ d\mathbf{%
\gamma }+\frac{d}{dt}\left( \int_{\Omega }\,|D(\mathbf{p}_{n})|^{2}\,d%
\mathbf{x}\right) \\
&=&\int_{\Gamma }(\alpha +a)(\mathbf{p}_{n}\cdot \bm{\tau })\partial _{t}(%
\mathbf{p}_{n}\cdot \bm{\tau })\ d\mathbf{\gamma }+\frac{d}{dt}\left(
\int_{\Omega }\,|D(\mathbf{p}_{n})|^{2}\,d\mathbf{x}\right) \\
&=&\frac{1}{2}\int_{\Gamma }(\alpha +a)\partial _{t}|(\mathbf{p}_{n}\cdot %
\bm{\tau })|^{2}\ d\mathbf{\gamma }+\frac{d}{dt}\left( \int_{\Omega }\,|D(%
\mathbf{p}_{n})|^{2}\,d\mathbf{x}\right) \\
&=&\frac{d}{dt}\left[ \int_{\Gamma }(\alpha +a)\frac{|\mathbf{p}_{n}|^{2}}{2}%
\ d\mathbf{\gamma }+\int_{\Omega }\,|D(\mathbf{p}_{n})|^{2}\,d\mathbf{x}%
\right] \\
&&-\int_{\Gamma }(\partial _{t}\alpha +\partial _{t}a)\frac{|\mathbf{p}%
_{n}|^{2}}{2}\ d\mathbf{\gamma }.
\end{eqnarray*}%
Therefore%
\begin{align}
& -\frac{d}{dt}\left[ \int_{\Gamma }(\alpha +a)\frac{|\mathbf{p}_{n}|^{2}}{2}%
\ d\mathbf{\gamma }+\int_{\Omega }\,|D(\mathbf{p}_{n})|^{2}\,d\mathbf{x}%
\right] +\int_{\Omega }|A\mathbf{p}_{n}|^{2}\,d\mathbf{x}  \notag \\
& =-\int_{\Gamma }(\partial _{t}\alpha +\partial _{t}a)\frac{|\mathbf{p}%
_{n}|^{2}}{2}\ d\mathbf{\gamma }+\int_{\Omega }\mathbf{U\cdot }A\mathbf{p}%
_{n}d\mathbf{x}+\int_{\Omega }\left( 2D(\mathbf{p}_{n})\mathbf{y}\right)
\cdot A\mathbf{p}_{n}d\mathbf{x}  \notag \\
& =I_{1}+I_{2}+I_{3}\mathbf{.}  \label{a2}
\end{align}%
Let us estimate the terms $I_{1},$ $I_{2}$ and $I_{3}$. We have%
\begin{align*}
I_{1}& =C\Vert \partial _{t}\alpha +\partial _{t}a\Vert _{L_{\infty }(\Gamma
)}\Vert \mathbf{p}_{n}\Vert _{L_{2}(\Gamma )}^{2} \\
& \leqslant C\Vert \partial _{t}\alpha +\partial _{t}a\Vert _{L_{\infty
}(\Gamma )}||\mathbf{p}_{n}||_{L_{2}(\Omega )}||\nabla \mathbf{p}%
_{n}||_{L_{2}(\Omega )}\in L_{1}(0,T),
\end{align*}%
uniformly bounded on $n=1,2,....$\ by the hypothesis and \eqref{estp}. We
also have
\begin{equation*}
I_{2}=\int_{\Omega }\mathbf{U\cdot }A\mathbf{p}_{n}d\mathbf{x}\leqslant
\Vert \mathbf{U}\Vert _{L_{2}(\Omega )}\Vert A\mathbf{p}_{n}\Vert
_{L_{2}(\Omega )}\leqslant C\Vert \mathbf{U}\Vert _{L_{2}(\Omega )}^{2}+%
\frac{1}{4}\Vert A\mathbf{p}_{n}\Vert _{L_{2}(\Omega )}^{2}.
\end{equation*}%
and%
\begin{equation*}
I_{3}\leqslant C\Vert \mathbf{y}\Vert _{L_{6}(\Omega )}\Vert D(\mathbf{p}%
_{n})\Vert _{L_{3}(\Omega )}\,\Vert A\mathbf{p}_{n}\Vert _{L_{2}(\Omega )}.
\end{equation*}%
Using Gagliardo-Nirenberg-Sobolev's inequality (\ref{LI}) with $q=6$ and
with $q=3,$ respectively,
\begin{equation*}
\Vert \mathbf{y}\Vert _{L_{6}(\Omega )}\leqslant C(\Vert \mathbf{y}\Vert
_{L_{2}(\Omega )}^{1/3}\Vert \nabla \mathbf{y}\Vert _{L_{2}(\Omega
)}^{2/3}+\Vert \mathbf{y}\Vert _{L_{2}(\Omega )})=f(t)\in L_{3}(0,T)\text{
\qquad by \eqref{uny},}
\end{equation*}%
\begin{equation*}
\Vert \mathbf{u}\Vert _{L_{3}(\Omega )}\leqslant C\left( \Vert \mathbf{u}%
\Vert _{L_{2}(\Omega )}^{2/3}\Vert \nabla \mathbf{u}\Vert _{L_{2}(\Omega
)}^{1/3}+\Vert \mathbf{u}\Vert _{L_{2}(\Omega )}\right) \qquad \text{for \ }%
\mathbf{u}=D(\mathbf{p}_{n}),
\end{equation*}%
we get%
\begin{equation*}
I_{3}\leqslant Cf(t)\left( \Vert \nabla \mathbf{p}_{n}\Vert _{L_{2}(\Omega
)}^{2/3}\Vert A\mathbf{p}_{n}\Vert _{L_{2}(\Omega )}^{1/3}+\Vert \nabla
\mathbf{p}_{n}\Vert _{L_{2}(\Omega )}\right) \,\Vert A\mathbf{p}_{n}\Vert
_{L_{2}(\Omega )}.
\end{equation*}%
where we have used the inequality%
\begin{equation}
\Vert \mathbf{p}_{n}\Vert _{H^{2}(\Omega )}\leqslant C\Vert A\mathbf{p}%
_{n}\Vert _{L_{2}(\Omega )}  \label{ap}
\end{equation}%
which holds by the regular properties of the Stokes operator $A$ (see
Theorem 9 of \cite{amr2} and Theorem 2 of \cite{S73}). Therefore applying
Young's inequality (\ref{yi}) and Korn's inequality (\ref{Korn}) we derive%
\begin{equation*}
I_{3}\leqslant Ch_{1}(t)\Vert D(\mathbf{p}_{n})\Vert _{L_{2}(\Omega )}^{2}+%
\frac{1}{4}\Vert A\mathbf{p}_{n}\Vert _{L_{2}(\Omega )}^{2}
\end{equation*}%
with $h_{1}(t)=f^{3}(t)\in L_{1}(0,T).$

Therefore, the above deduced estimates for the terms $I_{1},$ $I_{2},$ $I_{3}
$\ and \eqref{a2} imply
\begin{align*}
-\frac{d}{dt}\biggl[\int_{\Gamma }(\alpha +a)\frac{|\mathbf{p}_{n}|^{2}}{2}\
d\mathbf{\gamma }+\int_{\Omega }\,|D(\mathbf{p}_{n})|^{2}\,d\mathbf{x}\biggr]%
& +\frac{1}{2}\Vert A\mathbf{p}_{n}\Vert _{L_{2}(\Omega )}^{2}\leqslant
C\Vert \mathbf{U}\Vert _{L_{2}(\Omega )}^{2} \\
+& Ch_{2}(t)\left( \Vert D(\mathbf{p}_{n})\Vert _{L_{2}(\Omega
)}^{2}+1\right) ,
\end{align*}%
with some $h_{2}(t)\in L_{1}(0,T)$ depending only on the data %
\eqref{eq00sec12} of our problem \eqref{NSy}. Integrating this inequality
over the time interval $(t,T),$ we obtain
\begin{align*}
\left\Vert D\left( \mathbf{p}_{n}(t)\right) \right\Vert _{L_{2}(\Omega
)}^{2}& +\frac{1}{2}\int_{t}^{T}\Vert A\mathbf{p}_{n}(s)\Vert
_{L_{2}(0,T;L_{2}(\Omega ))}^{2}ds\leqslant C(\Vert \mathbf{U}\Vert
_{L_{2}(0,T;L_{2}(\Omega ))}^{2}) \\
& +C\int_{t}^{T}h_{2}(s)\left( \Vert D(\mathbf{p}_{n})\Vert _{L_{2}(\Omega
)}^{2}+1\right) ds-\left[ \int_{\Gamma }(\alpha +a)\frac{|\mathbf{p}_{n}|^{2}%
}{2}\ d\mathbf{\gamma }\right] .
\end{align*}%
Finally, with the help of the Korn inequality, we deduce
\begin{align*}
I& =\left[ \int_{\Gamma }(\alpha +a)\frac{|\mathbf{p}_{n}|^{2}}{2}\ d\mathbf{%
\gamma }\right] \leqslant C\Vert \alpha +a\Vert _{L_{\infty }(\Gamma )}\Vert
\mathbf{p}_{n}\Vert _{L_{2}(\Gamma )}^{2} \\
& \leqslant C\Vert \alpha +a\Vert _{L_{\infty }(\Gamma )}||\mathbf{p}%
_{n}||_{L_{2}(\Omega )}||\nabla \mathbf{p}_{n}||_{L_{2}(\Omega )} \\
& \leqslant h_{3}(t)\Vert \mathbf{p}_{n}\Vert _{L_{2}(\Omega )}^{2}+\frac{1}{%
2}||D(\mathbf{p}_{n})||_{L_{2}(\Omega )}^{2},
\end{align*}%
where $h_{3}(t)=C\Vert \alpha +a\Vert _{L_{\infty }(\Gamma )}^{2}\in
L_{\infty }(0,T)$ by the hypothesis \eqref{eq00sec12}. Then we have the
Gronwall inequality
\begin{align*}
\left\Vert D\left( \mathbf{p}_{n}\right) \right\Vert _{L_{\infty
}(0,T;L_{2}(\Omega ))}^{2}& +\int_{t}^{T}\Vert A\mathbf{p}_{n}\Vert
_{L_{2}(0,T;L_{2}(\Omega ))}^{2}ds\leqslant C(\Vert \mathbf{U}\Vert
_{L_{2}(0,T;L_{2}(\Omega ))}^{2})  \notag \\
& +C\int_{t}^{T}h_{2}(s)\left( \Vert D(\mathbf{p}_{n})\Vert _{L_{2}(\Omega
)}^{2}+1\right) ds+h_{3}(t)\Vert \mathbf{p}_{n}\Vert _{L_{2}(\Omega )}^{2}
\label{estpt1}
\end{align*}%
which gives
\begin{equation}
\left\Vert D\left( \mathbf{p}_{n}\right) \right\Vert _{L_{\infty
}(0,T;L_{2}(\Omega ))}^{2}+\Vert A\mathbf{p}_{n}\Vert _{L_{2}(\Omega
_{T})}^{2}\leqslant C(\Vert \mathbf{U}\Vert _{L_{2}(\Omega _{T})}^{2}+1)
\label{estpt1}
\end{equation}%
where $C$ is a constant only depending on the data. \ Hence \ \eqref{ap}
implies
\begin{equation}
\Vert \mathbf{p}_{n}\Vert _{L_{2}(0,T;H^{2}(\Omega ))}\leqslant C\Vert A%
\mathbf{p}_{n}\Vert _{L_{2}(0,T;L_{2}(\Omega ))}^{2}\leqslant C.  \label{sa}
\end{equation}

Moreover we can take $\boldsymbol{\psi }=\widetilde{\mathbf{e}}_{j}$ in (\ref%
{help}), multiply it by $\frac{d(s_{j}^{(n)}(t))}{dt}$ and summing on $%
j=1,...,n,$ then we deduce that
\begin{eqnarray*}
\Vert \partial _{t}\mathbf{p}_{n}\Vert _{L_{2}(\Omega )}^{2} &=&\int_{\Omega
}|\partial _{t}\mathbf{p}_{n}|^{2}\boldsymbol{\ }d\mathbf{x}=-\int_{\Omega
}\left( 2D(\mathbf{p}_{n})\mathbf{y}+\triangle \mathbf{p}_{n}+\mathbf{U}%
\right) \cdot \partial _{t}\mathbf{p}_{n}\boldsymbol{\ }d\mathbf{x} \\
&\leqslant &C\Vert \partial _{t}\mathbf{p}_{n}\Vert _{L_{2}(\Omega )}\left(
\Vert D(\mathbf{p}_{n})\Vert _{L_{4}(\Omega )}\Vert \mathbf{y}\Vert
_{L_{4}(\Omega )}+\Vert \triangle \mathbf{p}_{n}\Vert _{L_{2}(\Omega
)}+\Vert \mathbf{U}\Vert _{L_{2}(\Omega )}\right) .
\end{eqnarray*}%
Since%
\begin{equation*}
\Vert \mathbf{y}\Vert _{L^{4}(\Omega )}\leqslant \left( ||\mathbf{y}%
||_{L_{2}(\Omega )}^{1/2}||\nabla \mathbf{y}||_{L_{2}(\Omega )}^{1/2}+||%
\mathbf{y}||_{L_{2}(\Omega )}\right) ,
\end{equation*}%
\begin{equation*}
\Vert D(\mathbf{p}_{n})\Vert _{L^{4}(\Omega )}\leqslant \left( ||D(\mathbf{p}%
_{n})||_{L_{2}(\Omega )}^{1/2}||\nabla \left( D(\mathbf{p}_{n})\right)
||_{L_{2}(\Omega )}^{1/2}+||D(\mathbf{p}_{n})||_{L_{2}(\Omega )}\right)
\end{equation*}%
\ by \eqref{ra}, we obtain
\begin{equation}
\Vert \partial _{t}\mathbf{p}_{n}\Vert _{L_{2}(\Omega _{T})}^{2}\leqslant
C\left( \int_{0}^{T}\Vert D(\mathbf{p}_{n})\Vert _{L_{4}(\Omega )}^{2}\Vert
\mathbf{y}\Vert _{L_{4}(\Omega )}^{2}dt+\Vert \mathbf{U}\Vert _{L_{2}(\Omega
_{T})}^{2}\right) \leqslant C  \label{pa}
\end{equation}%
for the constant $C$ being independent of $n$ \ by \eqref{uny}, %
\eqref{estpt1} and \eqref{sa}.

Therefore \eqref{estpt1}, \eqref{sa} and \eqref{pa} imply that there exists
a suitable subsequence of $\left\{ \mathbf{p}_{n}\right\} ,$ such that%
\begin{eqnarray*}
\mathbf{p}_{n} &\rightarrow &\mathbf{p}\qquad \mbox{ weakly in
}\ L_{2}(0,T;H^{2}(\Omega ))\cap L_{\infty }(0,T;V),\qquad  \\
\partial _{t}\mathbf{p}_{n} &\rightarrow &\partial _{t}\mathbf{p}\qquad
\mbox{ weakly in
}\ L_{2}(\Omega _{T}), \\
\mathbf{p}_{n} &\rightarrow &\mathbf{p}\qquad \mbox{ strongly in
}\ L_{2}(\Omega _{T}).
\end{eqnarray*}%
Taking the limit on $n\rightarrow \infty $ in \eqref{help}, we derive that%
\begin{equation*}
\mathbf{p}\in L_{2}(0,T;H^{2}(\Omega ))\cap L_{\infty }(0,T;V),\qquad
\partial _{t}\mathbf{p}\in L_{2}(\Omega _{T})
\end{equation*}%
satisfies the equality%
\begin{eqnarray*}
\int_{\Omega }\left( \partial _{t}\mathbf{p}+2D(\mathbf{p})\mathbf{y}%
+\triangle \mathbf{p}+\mathbf{U}\right) \cdot \boldsymbol{\psi \ }d\mathbf{x}
&=&0,\text{\qquad }\forall \boldsymbol{\psi }\in V,\text{ a.e. in }(0,T), \\
\mathbf{p}(T) &=&\mathbf{0}.
\end{eqnarray*}%
and has the regularity \eqref{up}. Hence $\mathbf{p}$ fulfills the system %
\eqref{adjoint} in the usual sense. Moreover, reasoning as in Proposition
1.2, p. 182 of \cite{tem}, we derive that $\pi \in L_{2}(0,T;H^{1}(\Omega )).
$
\end{proof}

\bigskip

\section{Duality property}

\label{sec7}\setcounter{equation}{0}

\bigskip

In the next proposition we demonstrate the \textit{duality} property for the
solution $\mathbf{z}$ of the linearized equation (\ref{linearized}) and the
adjoint pair $(\mathbf{p},\pi ),$ \ being the solution of (\ref{adjoint}).

\begin{proposition}
The solution $\mathbf{z}$ of \ the\ system (\ref{linearized}) and the
solution $(\mathbf{p},\pi )$ of the adjoint system (\ref{adjoint}) verify
the following \textit{\ duality relation}%
\begin{equation}
\int_{\Omega _{T}}\mathbf{z}\cdot U\,d\mathbf{x}dt=\int_{\Gamma _{T}}\left\{
g(\mathbf{p}\cdot \bm{\tau })+f\left[ \pi -\left( \mathbf{p}\cdot \mathbf{y}%
\right) -2\left( D(\mathbf{p})\mathbf{n}\right) \cdot \mathbf{n}\,\right]
\right\} \,d\mathbf{\gamma }dt  \label{duality}
\end{equation}
\end{proposition}

\begin{proof}
If we multiply (\ref{adjoint})$\ $by $\mathbf{z}$, we have%
\begin{equation}
\int_{\Omega _{T}}\mathbf{z}\cdot Ud\mathbf{x}dt=\int_{\Omega _{T}}\mathbf{z}%
\cdot \left\{ -\partial _{t}\mathbf{p}-2D(\mathbf{p})\mathbf{y}+\nabla \pi
-\Delta \mathbf{p}\right\} ~d\mathbf{x}dt.  \label{duality10}
\end{equation}%
The integration by parts\ gives the following three relations
\begin{equation*}
\int_{\Omega }\mathbf{z}\cdot \nabla \pi \,d\mathbf{x}=\int_{\Gamma }\left(
\mathbf{z}\cdot \mathbf{n}\right) \pi \,d\mathbf{\gamma },
\end{equation*}%
\begin{align*}
-\int_{\Omega }\mathbf{z}\cdot \left( 2D(\mathbf{p})\mathbf{y}\right) \ d%
\mathbf{x}& =\int_{\Omega }\left[ \left( \mathbf{y}\cdot \nabla \right)
\mathbf{z}+\left( \mathbf{z}\cdot \nabla \right) \mathbf{y}\right] \cdot
\mathbf{p}\,d\mathbf{x} \\
& -\int_{\Gamma }\left( \left( \mathbf{y}\cdot \mathbf{n}\right) \left(
\mathbf{p}\cdot \mathbf{z}\right) +\left( \mathbf{z}\cdot \mathbf{n}\right)
\left( \mathbf{p}\cdot \mathbf{y}\right) \right) \,d\mathbf{\gamma }
\end{align*}%
and%
\begin{equation*}
-\int_{\Omega }\mathbf{z}\cdot \Delta \mathbf{p}\,d\mathbf{x}=-\int_{\Gamma
}2\left( D(\mathbf{p})\mathbf{n}\right) \cdot \mathbf{z}\,d\mathbf{\gamma }%
+\int_{\Omega }2\,D(\mathbf{p}):D(\mathbf{z})\,d\mathbf{x}.
\end{equation*}%
by (\ref{integrate}). \ Substituting these three relations in (\ref%
{duality10}), we obtain%
\begin{align*}
& \int_{\Omega _{T}}\mathbf{z}\cdot Ud\mathbf{x}dt=\int_{\Omega _{T}}\left\{
-\mathbf{z}\cdot \partial _{t}\mathbf{p}+\left[ \left( \mathbf{y}\cdot
\nabla \right) \mathbf{z}+\left( \mathbf{z}\cdot \nabla \right) \mathbf{y}%
\right] \cdot \mathbf{p}+2\,D(\mathbf{p}):D(\mathbf{z})\right\} ~d\mathbf{x}%
dt \\
& +\int_{\Gamma _{T}}\left[ \left( \mathbf{z}\cdot \mathbf{n}\right) \pi
-\left\{ \left( \mathbf{y}\cdot \mathbf{n}\right) \left( \mathbf{p}\cdot
\mathbf{z}\right) +\left( \mathbf{z}\cdot \mathbf{n}\right) \left( \mathbf{p}%
\cdot \mathbf{y}\right) +\left( 2D(\mathbf{p})\mathbf{n}\right) \cdot
\mathbf{z}\right\} \right] \,d\mathbf{\gamma }dt.
\end{align*}

By another hand if we take $\boldsymbol{\psi }=\mathbf{p}\in L_{2}(0,T;V)$ \
in (\ref{duality10}), we have%
\begin{eqnarray*}
\int_{\Omega _{T}}\{-\mathbf{z}\cdot \partial _{t}\mathbf{p}+\left[ (\mathbf{%
z}\cdot \nabla )\mathbf{y}+(\mathbf{y}\cdot \nabla )\mathbf{z}\right] \cdot
\mathbf{p}+2\,D(\mathbf{z}) &:&D(\mathbf{p})\,\}d\mathbf{x}dt \\
&=&\int_{\Gamma _{T}}(g-\alpha (\mathbf{z}\cdot {\bm{\tau })})(\mathbf{p}%
\cdot {\bm{\tau })}\,d\mathbf{\gamma }\,dt,
\end{eqnarray*}%
that implies%
\begin{align*}
& \int_{\Omega _{T}}\mathbf{z}\cdot Ud\mathbf{x}dt=\int_{\Gamma _{T}}[g(%
\mathbf{p}\cdot {\bm{\tau })}-\alpha (\mathbf{z}\cdot {\bm{\tau })}(\mathbf{p%
}\cdot {\bm{\tau })}\, \\
& +\left( \mathbf{z}\cdot \mathbf{n}\right) \pi -\left\{ \left( \mathbf{y}%
\cdot \mathbf{n}\right) \left( \mathbf{p}\cdot \mathbf{z}\right) +\left(
\mathbf{z}\cdot \mathbf{n}\right) \left( \mathbf{p}\cdot \mathbf{y}\right)
+\left( 2D(\mathbf{p})\mathbf{n}\right) \cdot \mathbf{z}\right\} \,]\ d%
\mathbf{\gamma }dt.
\end{align*}

Accounting the boundary conditions for $\mathbf{y}$, $\ \mathbf{z}$ and $%
\mathbf{p}$%
\begin{eqnarray*}
\left( \mathbf{y}\cdot \mathbf{n}\right)  &=&a,\qquad \left( \mathbf{z}\cdot
\mathbf{n}\right) =f,\qquad \left( \mathbf{p}\cdot \mathbf{n}\right) =0, \\
\left( 2D(\mathbf{p})\mathbf{n}\right) \cdot \bm{\tau } &=&-(a+\alpha )(%
\mathbf{p}\cdot \bm{\tau }),
\end{eqnarray*}%
we obtain%
\begin{align*}
& \int_{\Omega _{T}}\mathbf{z}\cdot Ud\mathbf{x}dt=\int_{\Gamma _{T}}g(%
\mathbf{p}\cdot {\bm{\tau })}-\alpha (\mathbf{p}\cdot {\bm{\tau })}(\mathbf{z%
}\cdot {\bm{\tau })}+f\pi \, \\
& -\{a(\mathbf{p}\cdot {\bm{\tau })}(\mathbf{z}\cdot {\bm{\tau })}+f(\mathbf{%
p}\cdot \mathbf{y}{)} \\
& +\left( \left( 2D(\mathbf{p})\mathbf{n}\right) \cdot \mathbf{n}\right)
f-(a+\alpha )(\mathbf{p}\cdot \bm{\tau })(\mathbf{z}\cdot {\bm{\tau })}\}\,d%
\mathbf{\gamma }dt \\
& =\int_{\Gamma _{T}}g(\mathbf{p}\cdot {\bm{\tau })}+f\pi -\{f(\mathbf{p}%
\cdot \mathbf{y}{)}+\left( \left( 2D(\mathbf{p})\mathbf{n}\right) \cdot
\mathbf{n}\right) f\}\,d\mathbf{\gamma }dt
\end{align*}%
that is%
\begin{equation*}
\int_{\Omega _{T}}\mathbf{z}\cdot Ud\mathbf{x}dt=\int_{\Gamma _{T}}\left\{ g(%
\mathbf{p}\cdot \bm{\tau })+f\left[ \pi -\left( \mathbf{p}\cdot \mathbf{y}%
\right) -2\left( D(\mathbf{p})\mathbf{n}\right) \cdot \mathbf{n}\,\right]
\right\} \ d\mathbf{\gamma }dt.
\end{equation*}%
which is the duality property (\ref{duality}).$\hfill $
\end{proof}

\section{Proof of the main results}

\label{sec8}\setcounter{equation}{0}

\subsection{Proof of Theorem \protect\ref{main_existence}}

Let us consider a minimizing sequence
\begin{equation*}
(a_{n},b_{n},\mathbf{y}_{a_{n},b_{n}})\in \mathcal{A}\times \left[ L_{\infty
}(0,T;L_{2}(\Omega ))\cap L_{2}(0,T;H^{1}(\Omega ))\right]
\end{equation*}%
of the cost functional $J$, namely
\begin{equation*}
\lim_{n}J(a_{n},b_{n},\mathbf{y}_{a_{n},b_{n}})=\inf (\mathcal{P}).
\end{equation*}

Since the sequence $(a_{n},b_{n})$ is bounded in $\mathcal{H}_{p}(0,T;\Gamma
)$ there exists a subsequence, still indexed by $n$, such that
\begin{equation*}
\left( a_{n},b_{n}\right) \rightarrow (a^{\ast },b^{\ast })\qquad \text{{%
weakly \ in}}\quad \mathcal{H}_{p}(0,T;\Gamma ).
\end{equation*}%
In addition, taking into account the estimate \eqref{uny}, we know that the
sequence $(\mathbf{y}_{a_{n},b_{n}})$ is uniformly bounded on the index $n$
in the space $L_{\infty }(0,T;H)\cap L_{2}(0,T;H^{1}(\Omega ))$, and $%
(\partial _{t}\mathbf{y}_{a_{n},b_{n}})$ is bounded in $L_{2}(0,T;H^{-1}(%
\Omega ))$, then there exists a subsequence, still indexed by $n$, such that
\begin{eqnarray*}
\mathbf{y}_{a_{n},b_{n}} &\rightharpoonup &\mathbf{y}^{\ast }\qquad
\mbox{ weakly in
}\ L_{\infty }(0,T;L_{2}(\Omega ))\cap L_{2}(0,T;H^{1}(\Omega )),\qquad  \\
\partial _{t}\mathbf{y}_{a_{n},b_{n}} &\rightharpoonup &\partial _{t}\mathbf{%
y}^{\ast }\qquad \mbox{ weakly in
}\ L_{2}(0,T;H^{-1}(\Omega )), \\
\mathbf{y}_{a_{n},b_{n}} &\rightarrow &\mathbf{y}^{\ast }\qquad
\mbox{ strongly in
}\ L_{2}(\Omega _{T}).
\end{eqnarray*}%
These convergence results allow to pass on the limit $n\rightarrow \infty $
in the variational formulation (\ref{cost}) for $\mathbf{y}_{a_{n},b_{n}}$
and in the equality (\ref{res1}), showing that $\mathbf{y}^{\ast }$
satisfies the integral equality%
\begin{align}
\int_{\Omega _{T}}\{-\mathbf{y}^{\ast }\cdot \boldsymbol{\psi }_{t}+& \left(
\left( \mathbf{y^{\ast }\cdot }\nabla \right) \mathbf{y}^{\ast }\right)
\cdot \boldsymbol{\psi }+2\,D(\mathbf{y}^{\ast }):D(\boldsymbol{\psi })\,\}d%
\mathbf{x}dt  \notag \\
=& \int_{\Gamma _{T}}(b^{\ast }-\alpha (y^{\ast }\cdot {\bm{\tau }}))(%
\boldsymbol{\psi }\cdot {\bm{\tau }})\,d\mathbf{\gamma }\,dt+\int_{\Omega }%
\mathbf{y}_{0}\cdot \boldsymbol{\psi }(0)\,d\mathbf{x,}
\end{align}%
which holds for any $\boldsymbol{\psi }\in H^{1}(0,T;V)$ with \ $\boldsymbol{%
\psi }(T)=0.$ Therefore $(a^{\ast },b^{\ast },\mathbf{y}^{\ast })$ is a
solution for the problem $(\mathcal{P}).$\vspace{2mm}\newline


\subsection{Proof of Theorem \protect\ref{main_1}}

Let $(a^{\ast },b^{\ast },\mathbf{y}^{\ast })$ be a solution of the problem $%
(\mathcal{P}).$ According to Theorem \ref{existence_state_y} and Proposition %
\ref{Gat}, for any $(a,b)\in \mathcal{H}_{p}(0,T;\Gamma )$ the corresponding
state equation (\ref{NSy}) has a unique solution $\mathbf{y}$ and the
control-to-state mapping
\begin{equation*}
(a,b)\rightarrow \mathbf{y}
\end{equation*}%
is the G\^{a}teaux differentiable at $(a^{\ast },b^{\ast })$. For $%
\varepsilon \in (0,1)$ and $(f,g)\in \mathcal{H}_{p}(0,T;\Gamma )$, let us
set $a_{\varepsilon }=a^{\ast }+\varepsilon (f-a^{\ast })$, $b_{\varepsilon
}=b^{\ast }+\varepsilon (g-b^{\ast })$ and $\mathbf{y}_{\varepsilon }$ the
corresponding state, being the solution of (\ref{res1}).

Since $(a^{\ast },b^{\ast },\mathbf{y}_{a^{\ast },b^{\ast }}^{\ast })$ is a
optimal solution and $(a_{\varepsilon },b_{\varepsilon },\mathbf{y}%
_{\varepsilon })$ is admissible, we have
\begin{equation*}
\lim_{\varepsilon \rightarrow 0}\frac{J(a_{\varepsilon },b_{\varepsilon },%
\mathbf{y}_{\varepsilon })-J(a^{\ast },b^{\ast },\mathbf{y}_{a^{\ast
},b^{\ast }}^{\ast })}{\varepsilon }\geq 0.
\end{equation*}%
By taking into account Proposition \ref{Gat}, we deduce that
\begin{equation}
\int_{\Omega _{T}}\mathbf{z}_{f,g}^{\ast }\cdot \left( \mathbf{y}^{\ast }-%
\mathbf{y}_{d}\right) \,d\mathbf{x}dt+\int_{\Gamma _{T}}\left( \lambda
_{1}\,a^{\ast }f\,+\lambda _{2}b^{\ast }g\right) \,d\mathbf{\gamma }dt\geq 0,
\label{zv}
\end{equation}%
where
\begin{equation*}
\mathbf{z}_{f,g}^{\ast }=\lim_{\varepsilon \rightarrow 0}\frac{\mathbf{y}%
_{\varepsilon }^{\ast }-\mathbf{y}^{\ast }}{\varepsilon }
\end{equation*}%
is the unique solution of the linearized equation
\begin{equation*}
\left\{
\begin{array}{ll}
\partial _{t}\mathbf{z}+(\mathbf{z}\cdot \nabla )\mathbf{y}^{\ast }+(\mathbf{%
y}^{\ast }\cdot \nabla )\mathbf{z}+\nabla \pi =\Delta \mathbf{z},\qquad div%
\mathbf{z}=0 & \quad \mbox{in}\ \Omega _{T},\vspace{2mm} \\
\mathbf{z}\cdot \mathbf{n}=f-a^{\ast },\quad \left[ 2D(\mathbf{z})\,\mathbf{n%
}+\alpha \mathbf{z}\right] \cdot \mathbf{\bm{\tau }}=g-b^{\ast }, & \quad %
\mbox{on}\ \Gamma _{T},\vspace{2mm} \\
\mathbf{z}(0)=0 & \quad \mbox{in}\ \Omega .%
\end{array}%
\right.
\end{equation*}%
On the other hand, taking $U=\mathbf{y}-\mathbf{y}_{d}$ and $y=y^{\ast }$ in
Proposition \ref{reg_extra}, we shows the existence of the adjoint state
pair $(\mathbf{p}^*, \pi^*)$ such that
\begin{equation*}
\mathbf{p}^*\in C([0,T];L_{2}(\Omega ))\cap L_2(0,T; H^{2}(\Omega )),\qquad
\pi^* \in L_{2}(0,T;H^{1}(\Omega ))
\end{equation*}
that verifies the equation (\ref{MR2}). Moreover, considering $\mathbf{z}=%
\mathbf{z}_{f,g}^{\ast }$ and $\mathbf{U}=\mathbf{y}^{\ast }-\mathbf{y}_{d}$
\ in the duality property (\ref{duality}), we have%
\begin{align*}
& \int_{\Omega _{T}}\mathbf{z}_{f,g}^{\ast }\cdot \left( \mathbf{y}^{\ast }-%
\mathbf{y}_{d}\right) \,d\mathbf{x}dt \\
& =\int_{\Gamma _{T}}\left\{ (f-a^{\ast })\left[ \pi +(\mathbf{p}^{\ast
}\cdot \mathbf{y}^{\ast })+2\left( \mathbf{n}\cdot D(\mathbf{p}^{\mathbf{%
\ast }})\right) \cdot \mathbf{n}\right] +(g-b^{\ast })\left( \mathbf{p}%
^{\ast }\cdot \mathbf{\bm{\tau }}\right) \right\} \,d\mathbf{\gamma }dt.
\end{align*}%
As a direct consequence of this equality and (\ref{zv}), we obtain the
necessary optimality condition (\ref{zvMR2}). \vspace{2mm}\newline

\textbf{Acknowledgment} The work of F. Cipriano was partially supported by
the Funda\c c\~ao para a Ci\^encia e a Tecnologia (Portuguese Foundation for
Science and Technology) through the project UID/MAT/00297/2013 (Centro de
Matem\'atica e Aplica\c c\~oes).


\vspace{1pt}

\vspace{1pt}

\begin{thebibliography}{99}
\bibitem{amr} \textsc{Amrouche C., Rodriguez-Belido M.A.}, \emph{On the
regularity for the Laplace equation and the Stokes system.} Monografias de
la Real Academia de Ciencias de Zaragoza, 1-20, \textbf{38} (2012).

\bibitem{amr2} \textsc{Amrouche C., Rejaiba A.}, Stationary Stokes equations
with friction slip boundary conditions. Monograf\'{\i}as Matem\'{a}ticas,
Garc\'{\i}a de Galdeano, 23--32, \textbf{39} (2014).

\bibitem{arn} \textsc{Arnal D., Juillen J.C., Reneaux J., Gasparian G.},
\emph{Effect of wall suction on leading edge contamination.} Aerospace
Science and Technology, 505--517, \textbf{8} (1997).

\bibitem{aub} \textsc{Aubin J.-P.,} \emph{Approximation of Elliptic
Boundary-Value Problems. }Wiley-Interscience, New York - Toronto, 1972.

\bibitem{bla} \textsc{Black T.L., Sarnecki A.J.,} The Turbulent Boundary
Layer with Suction or Injection. Aeronautical Research Council Reports and
Memoranda, N. 3387 (October, 1958), London, (1965).

\bibitem{bra} \textsc{Braslow A.L.}, A History of Suction-Type Laminar-Flow
Control with Emphasis on Flight Research. NASA History Division (1999).

\bibitem{buc} \textsc{Bucur D., Feireisl E., Necasova S.}, Boundary Behavior
of Viscous Fluids: influence of wall roughness and friction-driven Boundary
Conditions. Arch. Rational Mech. Anal., 117--138, \textbf{197} (2010).

\bibitem{BI06} \textsc{Busuioc A.V., Iftimie D.}, \emph{A non-Newtonian
fluid with Navier boundary conditions}, J. Dynam. Diff. Eq., 357-379,
\textbf{18} (2006).

\bibitem{BILN12} \textsc{Busuioc A. V., Iftimie D., Lopes Filho M.C.,
Nussenzveig Lopes H.J.}, \emph{Incompressible Euler as a limit of complex
fluid models with Navier boundary conditions}, J. Differential Equations,
624-640, \textbf{252} (2012).

\bibitem{clop} \textsc{Clopeau T., Mikelic A., Robert R.,} \emph{On the
vanishing viscosity limit for the 2D incompressible Navier-Stokes equations
with the friction type boundary conditions.} Nonlinearity, 1625--1636,
\textbf{11} (1998).

\bibitem{C} \textsc{Chemetov N.V., Antontsev S.N.,} \emph{Euler equations
with non-homogeneous Navier slip boundary condition.} Physica D: Nonlinear
Phenomena, 92--105, \textbf{237}, n. 1 (2008).

\bibitem{CC1} \textsc{Chemetov N.V., Cipriano F.,} \emph{Well-posedness of
stochastic second grade fluids.} Submitted.

\bibitem{CC2} \textsc{Chemetov N.V., Cipriano F.,} \emph{Optimal control for
two-dimensional stochastic second grade fluids.} Submitted.

\bibitem{CC3} \textsc{Chemetov N.V., Cipriano F.,} \emph{Boundary layer
problem: Navier-Stokes equations and Euler equations}. Nonlinear Analysis:
Real World Applications, 2091--2104, \textbf{14}, n. 6 (2013).

\bibitem{CC4} \textsc{Chemetov N.V., Cipriano F.,} \emph{The Inviscid Limit
for the Navier--Stokes Equations with Slip Condition on Permeable Walls}.
Journal of Nonlinear Science, 731--750, \textbf{23}, n.5 (2013).

\bibitem{CC5} \textsc{Chemetov N.V., Cipriano F., }\emph{Inviscid limit for
Navier--Stokes equations in domains with permeable boundaries}. Applied
Math. Letters, 6--11, \textbf{33} (2014).

\bibitem{CC6} \textsc{Chemetov N.V., Cipriano F., Gavrilyuk S.,} \emph{%
Shallow water model for the lake with friction and penetration}.
Mathematical Methods in the Applied Sciences, 687--703, \textbf{33}, n.6
(2010).

\bibitem{C96} \textsc{Coron J.M.}, \emph{On the controllability of the $2D$
incompressible Navier-Stokes equations with the Navier-slip boundary
conditions}, ESAIM Control Optim. Calc. Var., 35-75, \textbf{1} (1996).

\bibitem{evans} \textsc{Evans L.C.,} \emph{Partial Differential Equations.}
AMS, Graduate Studies in Mathematics, v. 19. 1998.

\bibitem{cal} \textsc{Gilbarg D., Trudinger N.S.,} \emph{Elliptic Partial
Differential Equations.} Springer-Verlag, Berlin Heidelberg New-York (2001).

\bibitem{gir} \textsc{Girault V., Raviart P.-A.,} \emph{Finite Element
Methods for Navier-Stokes equations,} Theory and Algorithms.
Springer-Verlag, Berlin Heidelberg New-York (1986).

\bibitem{ghs} \textsc{Gunzburger M., Hou L., Svobodny T.,} \emph{Analysis
and finite element approximation of optimal control problems for the
stationary Navier-Stokes equations with Dirichlet controls,} Mod\'el. Math.
Anal. Num., 711--748, \textbf{25} (1991).

\bibitem{gm} \textsc{Gunzburger M., Manservisi S.,} \emph{\ The velocity
tracking problem for Navier-Stokes flows with boundary control,} SIAM J.
Contr. Optim., 594--634, \textbf{30} (2000).

\bibitem{kel} \textsc{Kelliher J.,} \emph{Navier-Stokes equations with
Navier boundary conditions for a bounded domain in the plane.} SIAM J. Math.
Anal., 210--232, \textbf{38}, n. 1 (2006).

\bibitem{kel2} \textsc{Kelliher J., }\emph{The vanishing viscosity limit for
incompressible fluids in two dimensions.} PhD Thesis, The University of
Texas at Austin 2005

\bibitem{ja} \textsc{Jager W., Mikelic A.,} On the roughness-induced
effective boundary conditions for a viscous flow. J. Differential Equations,
96--122, \textbf{170} (2001).

\bibitem{kuf} \textsc{Kufner A., John O., Fu\v{c}\'{\i}k S.,} \ \emph{%
Function spaces.} Academia Publishing House of the Czechoslovak Academia of
Sciences, Prague, 1977.

\bibitem{lad} \textsc{Ladyzhenskaya O.A., Solonnikov V.A., Uraltseva N.N.,}
\emph{Linear and quasilinear equations of parabolic type.} \ Translations of
mathematical monographs, v. 23. Providence, R.I.: American Mathematical
Society (1968).

\bibitem{lions} \textsc{Lions P.-L.,} \emph{Mathematical topics in fluid
mechanics,} Vol. 1. The Clarendon Press Oxford University Press, New York
(1996).

\bibitem{LM} \textsc{Lions J. L., Magenes E.,} \emph{Non-Homogeneous
Boundary Value Problems and Applications,} \emph{Vol. I.} Springer-Verlag
Berlin Heidelberg New York 1972.

\bibitem{LionsMag} \textsc{Lions J.L., Magenes E.,} \emph{Probl\`{e}mes aux
limites non Homog\'{e}nes et Applications.} Vol. 2, Dunod, Paris (1968).

\bibitem{lieb} \textsc{Lieb E.H., Loss M.,} \emph{Analysis.} 2nd edition,
Graduate Studies in Math., vol. 14, AMS, Providence RI (2001).

\bibitem{mar} \textsc{Marshall L.A.,} Boundary-Layer Transition Results From
the F-16XL-2 Supersonic Laminar Flow Control Experiment.
NASA/TM-1999-209013, Dryden Flight Research Center Edwards, California
93523-0273, December, 1999.

\bibitem{nec} \textsc{Necas J.,} \emph{Direct Methods in the Theory of
Elliptic Equations.} Springer-Verlag, Berlin Heidelberg New-York (2010).

\bibitem{nir} \textsc{Nirenberg L.,} \emph{On elliptic partial differential
equations}. Ann. Scuola Norm. Sup. Pisa, 115--162, \textbf{13, }n. 3 (1959).

\bibitem{ol} \textsc{Oleinik O.A., Youssefian G.A.,} \emph{On the asymptotic
behaviour at infinity of solutions in linear elasticity.} Archive Rat. Mech.
Anal., 29-53, \textbf{78}, n. 1 (1982).

\bibitem{prie} \textsc{Priezjev N.V., Troian S.M.,} Influence of periodic
wall roughness on the slip behavior at liquid/solid interfaces:
molecular-scale simulations versus continuum predictions. J. Fluid Mech.,
25--46, \textbf{554} (2006).

\bibitem{pri} \textsc{Priezjev, N.V., Darhuber, A.A., Troian S.M.,} Slip
behavior in liquid films on surfaces of patterned wettability: Comparison
between continuum and molecular dynamics simulations. Phys. Rev. E 71,
041608 (2005).

\bibitem{tem} \textsc{Temam R.,} \emph{Navier-Stokes equations, Theory and
Numerical analysis.} AMS\ Chelsea Publishing, Providence, Rhode Island, 2001.

\bibitem{shu} \textsc{Schlichting H., Gersten K.,} Boundary-layer theory.
Springer-Verlag, Berlin Heidelberg New-York(2003).

\bibitem{sim} \textsc{Simon J.,} \emph{Compact Sets in the space }$%
L^{p}(0,T;B)$\emph{.} Annali di Matematica Pura ed Applicata, 65--96,
\textbf{146}, n. 1 (1986).

\bibitem{S73} \textsc{\v{S}\v{c}adilov V.E., Solonikov V.A.}, \emph{On a
boundary value problem for a stationary system of Navier-Stokes equations},
Proc. Steklov Inst. Math., 186-199, \textbf{125} (1973).

\bibitem{WR10} \textsc{Wachsmuth D., Roub\'{\i}\v{c}ek T.}, \emph{Optimal
control of incompressible non-Newtonian fluids}, Z. Anal. Anwend, 351-376,
\textbf{29} (2010).
\end{thebibliography}
\end{document}